\documentclass[11pt, a4paper]{amsart}
\usepackage{amssymb, array, amsmath, amscd, pdfpages, enumerate, amsthm, setspace,mathtools}
\usepackage[papersize={21cm,29.7cm},total={12.7cm,22.5cm},top=3.6cm,left=4.15cm]{geometry}

\usepackage[utf8]{inputenc}
\usepackage{amssymb}
\usepackage{bm}
\usepackage{overpic}
\usepackage[colorlinks=true,allcolors=magenta]{hyperref}
\usepackage{xcolor}
\usepackage{graphicx} 
\usepackage{ifthen}
\usepackage[capitalize,noabbrev]{cleveref}

\newcommand{\Dom}{D}
\newcommand{\eps}{\varepsilon}

\DeclareMathOperator{\re}{Re}

\DeclareMathOperator{\diag}{diag}

\newcommand*{\C}{{\mathbb{C}}}     
\newcommand*{\R}{{\mathbb{R}}}     
     
\newcommand*{\N}{{\mathbb{N}}}     
\newcommand*{\Hloc}[1]{H^{#1}_{loc}}
\newcommand{\conj}[1]{\overline{#1}}

\newcommand*{\Lin}{{\mathcal{L}}}   
\newcommand{\ran}{{\mathcal{R}}}   
\renewcommand{\ker}{{\mathcal{N}}}

\newcommand*{\abs}[1]{\lvert#1\rvert}
\newcommand*{\norm}[1]{\lVert#1\rVert}
\newcommand*{\set}[1]{\{#1\}}
\newcommand*{\setm}[2]{\{\,#1\mid#2\,\}}   
\newcommand*{\iprod}[2]{\langle#1,#2\rangle}    
\newcommand*{\Iprod}[2]{\left\langle#1,#2\right\rangle}    

\newcommand*{\ldelim}[2]{\csname#1l\endcsname#2}   
\newcommand*{\rdelim}[2]{\csname#1r\endcsname#2}   
\newcommand*{\mdelim}[2]{\csname#1m\endcsname#2}   
\newcommand*{\Set}[2][default]{\ifthenelse{\equal{#1}{default}}{\left\{#2\right\}}{\ldelim{#1}{\{}#2\rdelim{#1}{\}}}} 
\newcommand{\Setm}[3][big]{\ldelim{#1}{\{}\,#2\mdelim{#1}{|}#3\,\rdelim{#1}{\}}} 
\newcommand*{\Lp}[1][p]{L^{#1}}

\newcommand*{\Lploc}[1][p]{L^{#1}_{\text{loc}}}

  \newcommand{\pmat}[1]{\begin{bmatrix}#1\end{bmatrix}}
\newcommand{\pmatsmall}[1]{\begin{bsmallmatrix}#1\end{bsmallmatrix}}

\newcommand*{\List}[2][1]{\set{#1,\ldots,#2}}

\newcommand{\eq}[1]{\begin{align*}#1\end{align*}}
\newcommand{\eqn}[1]{\begin{align}#1\end{align}}

\newcommand{\gs}{\sigma}
\newcommand{\ga}{\alpha}
\newcommand{\gb}{\beta}
\renewcommand{\gg}{\gamma}
\newcommand{\gd}{\delta}
\newcommand{\gl}{\lambda}
\newcommand{\gw}{\omega}

\newcommand{\inv}{^{-1}}

\newcommand{\citel}[2]{\cite[#2]{#1}}

\renewcommand{\pmat}[1]{\begin{bmatrix}#1\end{bmatrix}}
\renewcommand{\pmatsmall}[1]{\begin{bsmallmatrix}#1\end{bsmallmatrix}}

\DeclareMathOperator{\blkdiag}{blockdiag}

\newcommand{\AB}{A\&B}
\renewcommand{\CD}{C\&D}

\newcommand{\id}{\mathrm{id}}

\newcommand{\Stot}{S}

\newcommand{\Aexo}{A_{\mathrm{exo}}}

\renewcommand{\L}{L}

\newcommand{\T}{\mathbb{T}}
\newcommand{\F}{\mathbb{F}}
\newcommand{\Pt}[1][t]{\mathbf{P}_{#1}}

\renewcommand{\Dom}{D}

\renewcommand{\AA}{\mathfrak{A}}
\newcommand{\BB}{\mathfrak{B}}
\newcommand{\CC}{\mathfrak{C}}

\renewcommand{\pmat}[1]{\begin{bmatrix}#1\end{bmatrix}}
\renewcommand{\pmatsmall}[1]{\begin{bsmallmatrix}#1\end{bsmallmatrix}}

\renewcommand{\ran}{\textup{Ran}}
\renewcommand{\ker}{\textup{Ker}}

\newcommand*{\dda}[3][1]{\ifthenelse{\equal{#1}{1}}{\frac{d#3}{d#2}}{\frac{d^{#1}#3}{d#2^{#1}}}}
\newcommand*{\ddb}[2][1]{\ifthenelse{\equal{#1}{1}}{\frac{d}{d#2}}{\frac{d^{#1}}{d#2^{#1}}}}
\newcommand*{\pd}[3][1]{\ifthenelse{\equal{#1}{1}}{\frac{\partial{#2}}{\partial{#3}}}{\frac{\partial^{#1}{#2}}{\partial#3^{#1}}}}
\newcommand*{\pdb}[2][1]{\ifthenelse{\equal{#1}{1}}{\frac{\partial}{\partial{#2}}}{\frac{\partial^{#1}}{\partial#2^{#1}}}}

\newcommand{\zinf}[1][0]{[#1,\infty)}

\newcommand{\zabl}[2]{[#1,#2)}

\newenvironment{SGORP}{\smallskip\noindent\textbf{Saturated Output Regulation Problem.}\it}{}
\newcommand{\SORP}{Saturated Output Regulation Problem}

\newcommand{\coeffset}{\Lambda}
\newcommand{\Ulim}{U_\phi}

\newcommand{\ureg}{u_{\mathrm{reg}}}
\newcommand{\yref}{y_{\mbox{\scriptsize\textit{ref}}}}
\newcommand{\wdist}{w_{\mbox{\scriptsize\textit{d}}}}
\newcommand{\wdistk}[1]{w_{\mbox{\scriptsize\textit{d},$#1$}}}

\newcommand{\rk}[1][k]{r_{#1}}

\newtheorem{theorem}{Theorem}[section]
\newtheorem{lemma}[theorem]{Lemma}
\newtheorem{proposition}[theorem]{Proposition}

\theoremstyle{definition}
\newtheorem{definition}[theorem]{Definition}
\newtheorem{assumption}[theorem]{Assumption}
\newtheorem{remark}[theorem]{Remark}

\numberwithin{equation}{section}

\crefname{proposition}{Proposition}{Propositions}
\Crefname{proposition}{Proposition}{Propositions}
\crefname{lemma}{Lemma}{Lemmas}
\Crefname{lemma}{Lemma}{Lemmas}
\crefname{corollary}{Corollary}{Corollaries}
\Crefname{corollary}{Corollary}{Corollaries}
\crefname{definition}{Definition}{Definitions}
\Crefname{definition}{Definition}{Definitions}
\crefname{remark}{Remark}{Remarks}
\Crefname{remark}{Remark}{Remarks}
\crefname{example}{Example}{Examples}
\Crefname{example}{Example}{Examples}
\crefname{assumption}{Assumption}{Assumptions}
\Crefname{assumption}{Assumption}{Assumptions}

\begin{document}

\title[Output Regulation for Passive Systems with Input Saturation]{Output Regulation for Impedance Passive Systems with Input Saturation}

\thispagestyle{plain}

\author[T.\ Govindaraj]{Thavamani Govindaraj}
\address[T.\ Govindaraj]{Faculty of Mathematics and Natural Sciences, Techniche Universität Ilmenau, Weimarer Stra{\ss}e 25, Ilmenau, 98693, Germany}
\email{thavam.rg@gmail.com}

\author[L.~Paunonen]{Lassi Paunonen}
\address[L.~Paunonen]{Mathematics Research Centre, Tampere University, P.O.~ Box 692, 33101 Tampere, Finland}
 \email{lassi.paunonen@tuni.fi}

\thanks{This research was supported by the Research Council of Finland grant 349002 and the German Research Foundation (DFG) grant 362536361.}

\begin{abstract}
We consider output tracking and disturbance rejection for abstract infinite-dimensi\-on\-al systems with input saturation.  We solve the control problem using a control law which consists of a stabilising error feedback term and a feedforward term based on the reference and disturbance signals.
We use our results to design control laws for output tracking and disturbance rejection for a two-dimensional heat equation and a one-dimensional wave equation with boundary control and observation.
\end{abstract}


\maketitle

\section{Introduction}

In this article we solve the problem of output tracking and disturbance rejection for a well-posed linear system $\Sigma$~\cite{Sta05book} affected by a saturation-type nonlinearity $\phi$ at the input (\cref{fig:ContrScheme}).
More precisely, the main goal in the \emph{output regulation problem} is to design a control input $u(t)$ in such a way that the \emph{tracking error} $e(t)=y(t)-\yref(t)$ between the output $y(t)$ of the system and a given reference signal $\yref(t)$ converges to zero as $t\to \infty$ in a suitable sense despite the external disturbance signal $\wdist(t)$.
Our control law is based on the knowledge of the signals $\yref $ and $\wdist$ and additional feedback from the tracking error $e$.

\begin{figure}[h!]
\begin{center}
\includegraphics[width=0.65\linewidth]{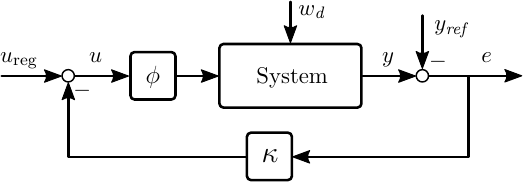}
\end{center}
\caption{The control scheme.}
\label{fig:ContrScheme}
\end{figure}

The solution of the output regulation problem 
for abstract linear systems is already understood very well~\cite{ByrLau00,NatGil14}. 
It is also well-known that in the case of finite-dimensional systems, very similar type of control design can be used to solve the control problem in the presence of input saturation~\citel{SabSto00book}{Sec.~3.1.1}. 
In this article we generalise the results in~\cite{SabSto00book} to solve the saturated output regulation problem for infinite-dimensional well-posed linear systems under suitable passivity assumptions of the well-posed system $\Sigma$.

The reference and disturbance signals are assumed to be of the form
\begin{subequations}
\label{eq:yrefwdist}
\eqn{
\yref(t) & = a_0+  \sum_{k=1}^{q} \left( a_k \cos(\gw_k t) + b_k \sin(\gw_k t) \right)\\
\wdist(t)  &= c_0+ \sum_{k=1}^{q} \left( c_k \cos(\gw_k t) + d_k \sin(\gw_k t) \right)
}
\end{subequations}
for some known frequencies $(\gw_k)_{k=1}^{q}\subset (0,\infty)$ and $\gw_0=0$
and known coefficients $(a_k)_k,(b_k)_k\subset U$, $(c_k)_k,(d_k)_k\subset U_d$.
Here $U$ is both the input and output space of the system, and $U_d$ is the input space of the external disturbance.
The control law solving the output regulation problem has the form
\eqn{
\label{eq:ControlLawintro}
 u(t)=-\kappa (y(t)-\yref(t)) + \ureg(t)
}
with
\eqn{
\label{eq:uregintro}
\ureg (t)
&=f_0
+\frac{1}{2}\sum_{k=1}^q
\Bigl[ (f_k+g_k)\cos(\gw_k t)  +i(f_k-g_k)\sin(\gw_k t) \Bigr].
}
The parameters $(f_k)_k\subset U$ and $(g_k)_k\subset U$ are computed based on the coefficients 
$(a_k)_k$, $(b_k)_k$, $(c_k)_k$, $(d_k)_k$ and the values of the transfer function of $\Sigma$ evaluated at the complex frequencies $\set{0}\cup \set{\pm i\gw_k}_{k=1}^q$ 
using explicit formulas~\eqref{eq:uregcoeff}. 
Even though the assumption that $\wdist$ is fully known seems restrictive, we show in Section~\ref{sec:ComputingTheControl} that the parameters of the control input 
 can in fact be measured from the output of a stabilised version of the system. 
In addition, especially in the linear setting the type of control design studied in the paper is also a building block in certain internal model based error feedback controllers~\cite{Imm07a,GuoMen20}.

In our main result in \cref{thm:ORPmain} we show that the control law~\eqref{eq:ControlLawintro} solves the output regulation problem under the following assumptions.
\begin{itemize}
\item[(a)] The system $\Sigma$ considered with input $u$, $\wdist=0$, and with output $y$ is \emph{impedance passive}~\citel{Sta02}{Def.~4.1} and
$\kappa\ge 0$ is such that the output feedback $u(t)=-\kappa y(t)$ achieves strong stability of the linear closed-loop system.
\item[(b)] The complex frequencies $\set{0}\cup \set{\pm i\gw_k}_k$ do not overlap with the transmission zeros of $\Sigma$.
\item[(c)] The coefficients $(a_k)_k,(b_k)_k,(c_k)_k,(d_k)_k$ of $\yref$ and $\wdist$ are such that $\ureg(t)$ is within the linear regime of the saturation function $\phi$ for all $t\ge 0$.
\end{itemize}
In~\citel{SabSto00book}{Sec.~3.1} the 
 output regulation problem was solved in the case of finite-dimensional systems with input saturation using \emph{low-and-high-gain} state feedback control. Our assumption of passivity in (a) allows us to simplify the control design by replacing the state feedback with error feedback, and by removing restrictions on the gain parameter $\kappa\ge 0$. 
The condition in (b) is a standard assumption also in linear output regulation~\citel{ByrLau00}{Sec.~V}.
Condition (c) describes the class of reference and disturbance signals which can be regulated under the saturation constraint, and this condition was shown to be nearly sharp in~\citel{SabSto00book}{Sec.~3.4}.

In the final part of the article we use our general results in 
 output tracking and disturbance rejection for concrete partial differential equation (PDE) systems with boundary control and boundary measurement. The models we consider are a  heat equation on a two-dimensional domain and a one-dimensional wave equation with spatially varying physical parameters.

Output regulation of infinite-dimensional systems and controlled PDEs have been studied 
extensively in the literature, see, e.g.,~\cite{ByrLau00,RebWei03,NatGil14,HamPoh10,DeuGab18,VanBri23} and references therein.
Infinite-dimensional system with saturation and other input nonlinearities
have been considered in~\cite{LogRya98, LogRya99, LogRya00, FliLog01, CouLog09, GilGui22, AstMar22, LorPauCDC23, VanLor25arxiv} using integral control in the case of constant reference and disturbance signals.
There are far fewer references on output tracking time-varying signals in infinite-dimensional and PDE systems. The only reference~\cite{ZenMen24} directly related to our work considered state and observer-based feedback control design for a flexible wing model with boundary control. 
Our results differ from~\cite{ZenMen24} especially in the sense that they are applicable to classes of abstract infinite-dimensional systems and classes of PDE models and we present a detailed proof of well-posedness.
Because of this, our results are novel in the setting of regulation of time-varying signals~\eqref{eq:yrefwdist} for abstract infinite-dimensional systems with input saturation.
Moreover, in the case of \emph{linear} systems (where $\phi=\id$) our main result in \cref{thm:ORPmain} also partially extends the previously known results in~\cite{NatGil14} from the class of regular linear systems to well-posed linear systems.

Preliminary version of \cref{thm:ORPmain} was presented in~\cite{GovPauIFAC23} for a single-input-single-output system with bounded input and output operators. In this paper we extend the results of~\cite{GovPauIFAC23} to the large classes of well-posed linear systems and abstract boundary control systems, which in particular include models of PDEs with boundary control and observation. In addition, we express the control law in a more accessible form based on the coefficients in~\eqref{eq:yrefwdist} and present comments on the computation of the control law.

The article is organised as follows. In Section~\ref{sec:prelim} we present our theoretical background and definitions. The definition of the main control problem and the results on output regulation for well-posed systems are presented in Section~\ref{sec:ORP}. In Section~\ref{sec:BCS} we present versions of our results for abstract boundary control systems. Finally, in Section~\ref{sec:PDEcases} we solve the saturated output tracking problem for a two-dimensional heat equation and a one-dimensional wave equation with boundary control and observation.

\section{Preliminaries}
\label{sec:prelim}

If $X$ and $Y$ are Banach spaces and $A:\Dom(A)\subset X\rightarrow Y$ is a linear operator we denote by $\Dom(A)$, $\ker(A)$ and $\ran(A)$ the domain, kernel, and range of $A$, respectively. The space of bounded linear operators from $X$ to $Y$ is denoted by $\Lin(X,Y)$ and we write $\Lin(X)$ for $\Lin(X,X)$. If \mbox{$A:X\rightarrow X$,} then $\gs(A)$
and $\rho(A)$ denote the spectrum
and the \mbox{resolvent} set of $A$, respectively. 
The inner product on a Hilbert space is denoted by $\iprod{\cdot}{\cdot}$ and all our Hilbert spaces are complex.
For $T\in \Lin(X)$ on a Hilbert space $X$ we define $\re T = \frac{1}{2}(T+T^\ast)$.
If $A: \Dom(A)\subset X\to X$ generates a a strongly continuous semigroup we denote by $X_{-1}$ the completion of the space $(X,\norm{(\gl_0-A)\inv \cdot}_X)$ where  $\gl_0\in\rho(A)$ is fixed. Then $A$ extends to an operator in $\Lin(X,X_{-1})$~\citel{Sta05book}{Sec.~3.6}, and we also denote this extension by $A$.
For $\tau>0$ and $u:\zinf\to U$  
we define $\Pt[\tau]u: [0,\tau]\to U$ as the truncation of the function $u$ to the interval $[0,\tau]$.
We denote $\C_+ = \{\gl\in \C \ | \ \re\gl>0\}$.

We consider a well-posed linear system $\Sigma = (\Sigma_t)_{t\ge 0} = (\T_t,\Phi_t,\Psi_t,\F_t)_{t\ge0}$
defined in the sense of~\citel{TucWei14}{Def.~3.1}
with input space $U$,  state space $X$,  and output space $Y$, where
 $ X$, $U$, and $Y$ are Hilbert spaces.
The mild state trajectory $x\in C(\zinf;X)$ and mild output $y\in\Lploc[2](0,\infty;Y)$
of $\Sigma$
corresponding to the initial state 
$x_0\in X$  and the input $u\in \Lploc[2](0,\infty;U)$ are defined by
\begin{subequations}
\label{eq:PrelimMildStateOut}
\eqn{
x(t) &= \mathbb{T}_t x_0 + \Phi_t\Pt u,\label{Eq_State}\\ 
\mathbf{P}_t y &= \Psi_t x_0 + \mathbb{F}_t\Pt u,\label{Eq_Output} 
}
\end{subequations}
for all $t\geq 0$.
The \emph{extended output map} $\Psi_\infty:X\to \Lploc[2](0,\infty;Y)$ and \emph{extended input-output map} $\F_\infty:\Lploc[2](0,\infty;U)\to \Lploc[2](0,\infty;Y)$ of $\Sigma$ are defined as in~\citel{TucWei14}{Sec.~3}. With this notation, the mild output $y$ has the alternative formula
 $y=\Psi_\infty x_0 + \F_\infty u$.

\begin{definition}\label{def:ImpPassive}
A well-posed system $\Sigma$ is \emph{impedance passive} if $Y=U$ and
if for every $x_0\in X$ and  $u\in \Lploc[2](0,\infty;U)$
the corresponding state trajectory $x$ and output $y$ defined by~\eqref{eq:PrelimMildStateOut} satisfy
\eq{
\Vert x(t)\Vert_X^2 - \Vert x_0\Vert_X^2
\leq 2\re\int_0^t\langle u(s),y(s)\rangle_U d s, \qquad t\geq 0.
}
\end{definition}

\emph{System nodes} are closely related to well-posed linear systems. 

\begin{definition}[\textup{\citel{Sta05book}{Def.~4.7.2}}]
\label{def:SysNode} 
Let $X$, $U$, and $Y$ be Hilbert spaces.
A closed operator 
\eq{
S:=\pmat{A\&B\\ C\&D} : \Dom(S) \subset X\times U \to X\times Y
}
is called a \emph{system node} on the spaces $(U,X,Y)$ if the following hold.
\begin{itemize}
\item The operator $A: \Dom(A)\subset X\to X$ defined by $Ax =A\&B \pmatsmall{x\\0}$ for $x\in \Dom(A)=\{x\in X\ | \ (x,0)^\top\in \Dom(S)\}$ generates a strongly continuous semigroup on $X$.
\item The operator $A\& B$ (with domain $\Dom(S)$) extends to  $[A,\ B]\in \Lin(X\times U,X_{-1})$ and
 $\Dom(S) = \{(x,u)^\top \in X\times U\ | \ Ax+Bu\in X\}$.
\end{itemize}
\end{definition}

The operator $A$ is called the \emph{semigroup generator} of $S$.
Definition~\ref{def:SysNode} implies that $C\& D\in \Lin(\Dom(S),Y)$, and we define the \emph{output operator} $C\in \Lin(\Dom(A),Y)$ of $S$ by $Cx= C\& D \pmatsmall{x\\0}$ for all $x\in \Dom(A)$.
Finally, the \emph{transfer function} 
$P:\rho(A)\to \Lin(U,Y)$ of a system node $S$ is defined
by
\eq{
P(\gl)u = C\& D \pmat{(\gl-A)\inv Bu\\ u}, \qquad u\in U, \ \gl\in\rho(A).
}

Every well-posed linear system $\Sigma = (\T,\Phi,\Psi,\F)$ 
is associated with a unique system node $S$~\citel{TucWei14}{Sec.~4}.
In particular, if $x_0\in X$ and $u\in \Hloc{1}(0,\infty;U)$ are such that $(x_0,u(0))^\top \in \Dom(S)$, then the mild state trajectory $x$ and output $y$ of $\Sigma$ defined by~\eqref{eq:PrelimMildStateOut} satisfy $x\in C^1(\zinf;X)$, $x(0)=x_0$, and $y\in \Hloc{1}(0,\infty;Y)$ and
\eqn{
\label{eq:PrelimSysnodeEqn}
\pmat{\dot x(t)\\y(t)} = S \pmat{x(t)\\ u(t)}, \qquad t\geq 0.
}
Because of this, we call a system node $S$ \emph{well-posed} 
if it is the system node of some well-posed linear system $\Sigma$.
The transfer function of a well-posed system node $S$ coincides with the transfer function of the associated well-posed linear system $\Sigma$ on every right half-plane $\setm{\gl\in\C}{\re\gl >r}$ which is contained in $\rho(A)$~\citel{Sta05book}{Lem.~4.7.5(iii)}. It follows from~\citel{Sta02}{Thm.~4.2} that the transfer function of an impedance passive well-posed system satisfies $\re P(\gl)\ge 0$ for all $\gl\in\C_+$.

In our controller design we encounter nonlinear systems of the form
\eqn{
\label{eq:SysGenInputs}
\pmat{\dot x(t)\\ y(t)} = \Stot \pmat{x(t)\\\phi(u_1(t)-\kappa y(t))\\\wdist(t)},
}
where $x(t)\in X$ and $y(t)\in U$ are the state and output, respectively,  $u_1(t)\in U$ and $\wdist(t)\in U_d$ are the external inputs, $S$ is a system node on $(U\times U_d,X,U)$,  and $\phi: U\to U$ is a continuous function.
In the following we define the classical and generalised solutions of~\eqref{eq:SysGenInputs}.

\begin{definition}
\label{def:SysNodeStates}
Let $\Stot$ be a system node on $(U\times U_d,X,U)$ and let $\kappa\ge 0$.
A tuple $(x,u_1,\wdist,y)$ is called a \emph{classical solution} of \eqref{eq:SysGenInputs} on $\zinf$ if
\begin{itemize}
    \item $x\in C^1(\zinf;X)$,
$u,y\in C(\zinf;U)$, and
 $\wdist\in C(\zinf;U)$
    \item $(x(t),\phi(u_1(t)-\kappa y(t)),\wdist(t))^\top\in D(\Stot)$ for all $t\geq 0$ and
     \eqref{eq:SysGenInputs} holds for every $t\geq 0$.
\end{itemize}
A tuple $(x,u_1,\wdist,y)$ is called a \emph{generalised solution} of~\eqref{eq:SysGenInputs} on $\zinf$ if 
\begin{itemize}
    \item $x\in C(\zinf;X)$, 
    $u,y\in\Lploc[2](0,\infty;U)$, and $\wdist\in\Lploc[2](0,\infty;U_d)$
    \item there exists a sequence $(x^k,u_1^k,\wdist^k,y^k)_{k=1}^\infty$ of classical solutions of~\eqref{eq:SysGenInputs} on $\zinf$ such that for every $\tau>0$ we have
\eq{
(\Pt[\tau] x^k,\Pt[\tau] u_1^k,\Pt[\tau]\wdist^k,\Pt[\tau] y^k)^\top\to (\Pt[\tau] x,\Pt[\tau] u_1,\Pt[\tau]\wdist,\Pt[\tau] y)^\top
}
as $k\to \infty$ 
in $C([0,\tau];X)\times \L^2(0,\tau;U) \times \L^2(0,\tau;U_d)\times \L^2(0,\tau;U)$.
\end{itemize}
\end{definition}

\section{Saturated Output Regulation}
\label{sec:ORP}

\subsection{Control Design for Saturated Output Regulation}
\label{sec:ORPresults}

We consider a control system of the form
\eqn{
\label{eq:SysMain}
\pmat{\dot x(t)\\ y(t)} = S \pmat{x(t)\\\phi(u(t))\\ \wdist(t)}, \qquad x(0)=x_0,
}
where $S=\pmatsmall{\AB\\\CD}$
is a well-posed system node on the Hilbert spaces $(U\times U_d,X,U)$.
We assume that $U=U_1\times U_2\times \cdots \times U_p$ where $U_k$ are Hilbert spaces for $k\in \List{p}$. We
define the nonlinear function $\phi:U\to U$ so that $\phi(u)=(\phi_1(u_1),\ldots,\phi_p(u_p))^\top$ for $u=(u_1,\ldots,u_p)^\top\in U$, where 
  $\phi_k:U_k\to U_k$ are defined using the parameters
 $(\rk[1],\ldots,\rk[p])^\top\subset U_1\times \cdots \times U_p$ and $\set{\gd_k}_{k=1}^p\subset (0,\infty)$
by
\eqn{
\label{eq:phisat}
\phi_k(u)=
r_k + \frac{\gd_k(u-\rk)}{\max \set{\gd_k,\norm{u-\rk}}}
=
\begin{cases}
u & \mbox{if } \norm{u-\rk}\le \gd_k
\\
 \rk + \frac{\gd_k(u-\rk)}{\norm{u-\rk}}  & \mbox{if } \norm{u-\rk}>\gd_k
\end{cases}
}
for $u\in U_k$ and $k\in \List{p}$.
We note that the case where $p=1$, $U=U_1=\C$ and $\rk[1]=0$ corresponds to the symmetric scalar saturation function $\phi$ which is linear when $\abs{u}\le \gd_1$. On the other hand, 
 if we define $\phi_k$ such that $U_k=\C$, $\gd_k=(b-a)/2$, and $\rk=(a+b)/2$, then in the case of real-valued signals the function $\phi_k$ is exactly the saturation function on the interval $[a,b]$.
We may now formulate our control problem.

\begin{SGORP}
Choose a control law $u(t)$ in such a way that the following hold.
\begin{itemize}
\item The 
system
 has a well-defined generalised solution for all initial states $x_0\in X$ and for all $\yref$ and $\wdist$ of the form~\eqref{eq:yrefwdist}.
\item With  $\yref (t)\equiv 0$ and $\wdist(t)\equiv 0$ the origin is a globally asymptotically stable equilibrium point of~\eqref{eq:SysMain},
 i.e., the generalised solutions corresponding to all initial states  $x_0\in X$ satisfy $\norm{x(t)}\to 0$ as $t\to\infty$.
\item 
There exists a 
non-empty
 set $\coeffset\subset U^{q+1}\times U^q\times  U_d^{q+1}\times U_d^q$
 such that for all $x_0\in X$ and for all $(a_k,b_k,c_k,d_k)_k\subset \coeffset$ in~\eqref{eq:yrefwdist}
 the 
 state $x$ is uniformly bounded and the tracking error $e=y-\yref$ converges to zero in a suitable sense as $t\to\infty$.
\end{itemize} 
\end{SGORP}

We note that the ``suitable sense'' of convergence
 of $e=y-\yref$
 depends on the properties of the system and the type of solutions considered. 
Even in the absence of saturation, the error corresponding to a generalised 
solution
 may fail to converge pointwise to zero, but may instead satisfy a weaker property, e.g., $e\in \Lp[2](0,\infty;U)$~\cite{NatGil14}.
We present a solution of the \SORP\ under the following assumptions.

\begin{assumption}
\label{ass:ORPass}
Let $X$, $U$, and $U_d$ be Hilbert spaces and
suppose that the following hold.
\begin{itemize}
\item[\textup{(a)}] $S=\pmatsmall{\AB\\ \CD}$ is a well-posed system node on $(U\times U_d,X,U)$ and
\eq{
\re \Iprod{\AB \pmatsmall{x\\u\\0}}{x}_X \leq \re \Iprod{\CD \pmatsmall{x\\u\\0}}{u }_U
}
for all $x\in X$ and $u\in U$ such that $(x,u,0)^\top \in \Dom(S)$.
Moreover, the transfer function $P=[P_c,P_d]:\rho(A)\to \Lin(U\times U_d,U)$ of $S$ satisfies
 $\re P_c(\gl)\geq c_\gl I$ for some $\gl,c_\gl>0$.
\item[\textup{(b)}] Either $\kappa=0$ and the semigroup generated by $A$ is strongly stable, 
\item[\textup{(b')}] 
or, alternatively, $\kappa>0$,
 $\norm{r_k}<\gd_k $ for all $k\in \List{p}$, and
the semigroup generated by
 $A^\kappa: \Dom(A^\kappa)\subset X\to X$ defined by 
\eq{
\Dom(A^\kappa) &= \Setm{x\in X}{\exists v\in U: \ (x,-\kappa v,0)^\top \in \Dom(S), \ v=\CD \pmatsmall{x\\-\kappa v\\0}}
}
and
$A^\kappa x = Ax +  B(-\kappa v(x),0)^\top$
(where $v(x)$ is the element $v$ in the definition of $\Dom(A^\kappa)$)
is strongly stable.
\end{itemize}
\end{assumption}

We note that the condition in \cref{ass:ORPass}(a) implies that the well-posed system $\Sigma$ considered with input $u$, with $\wdist=0$, and output $y$ is impedance passive in the sense of \cref{def:ImpPassive}. 

\begin{remark}
\label{rem:StabAssRemark}
If the semigroup generated by $A^\kappa$ in \cref{ass:ORPass}(b') is strongly stable with $i\R\subset \rho(A^\kappa)$ or exponentially stable,
then the same is true for all $\kappa>0$.
This can be shown using the properties in~\citel{Sta05book}{Sec.~7.4}, impedance passivity, and similar arguments as in the proof of~\citel{ChiPau23}{Lem.~2.11(c)}.
\end{remark}

To define the control law, we first note that~\citel{Sta02}{Cor.~6.1} and~\citel{Sta05book}{Lem. 7.2.6} imply that $\pmatsmall{-\kappa I\\0}$ is an admissible feedback operator for the well-posed system $\Sigma=(\T,\Phi,\Psi,\F)$, and therefore the resulting closed-loop system $\Sigma^\kappa=(\T^\kappa,\Phi^\kappa,\Psi^\kappa,\F^\kappa)$ is well-posed.
We note that by~\citel{Sta05book}{Def.~7.4.2} the generator of $\T^\kappa$ is exactly the operator $A^\kappa$ in \cref{ass:ORPass}(b'). We denote the transfer function of 
$\Sigma^\kappa$
 by $P^\kappa = [P_c^\kappa,P_d^\kappa]$, and note that for all $\gl\in\rho(A)\cap\rho(A_\kappa)$ we have
$P^\kappa(\gl) = [P_c^\kappa(\gl),P_d^\kappa(\gl)] = (I+\kappa P_c(\gl))\inv [P_c(\gl),P_d(\gl)]$, where $P=[P_c,P_d]$ is the transfer function of $\Sigma$.

Our main result below shows that 
 the \SORP\ is solved by the control law
 $u(t)=-\kappa (y(t)-\yref(t)) + \ureg(t)$, where
\eqn{
\label{eq:ureg}
\ureg (t)
&=f_0
+\frac{1}{2}\sum_{k=1}^q
\Bigl[ (f_k+g_k)\cos(\gw_k t)  +i(f_k-g_k)\sin(\gw_k t) \Bigr].
}
Here $f_k$ and $g_k$ are defined based on the transfer function $P^\kappa=[P_c^\kappa,P_d^\kappa]$  and the coefficients 
$(a_k)_k,(b_k)_k\subset U$ and $(c_k)_k,(d_k)_k\subset U_d$ of the reference disturbance signals as
\begin{subequations}
\label{eq:uregcoeff}
\eqn{
f_0 &=  P_c^\kappa(0)\inv ( (I-\kappa P_c^\kappa(0)) a_0-P_d^\kappa(0)c_0) \\
f_k &=   P_c^\kappa(i\gw_k)\inv \left[ (I-\kappa P_c^\kappa(i\gw_k))(a_k-ib_k) - P_d^\kappa(i\gw_k)(c_k-id_k) \right] \\
g_k &=   P_c^\kappa(-i\gw_k)\inv \left[ (I-\kappa P_c^\kappa(-i\gw_k))(a_k+ib_k) - P_d^\kappa(-i\gw_k)(c_k+id_k) \right]
\hspace{-1.5ex}
}
\end{subequations}
for $k\in \List{q}$.
If $a_0=0$ and $c_0=0$, we define $f_0=0$ even if $P_c^\kappa(0)$ is not invertible.
Detailed comments on computing the coefficients of $\ureg$ are presented in Section~\ref{sec:ComputingTheControl}.
We define
$\Ulim^\gd=\setm{(u_1,\ldots,u_p)^\top\in U_1\times\cdots\times U_p}{\norm{u_k-\rk}\le\gd_k-\gd, \ k=1,\ldots,p}$,
where $\set{\rk}_{k=1}^p$ and $\set{\gd_k}_{k=1}^p$ are as in the definition of $\phi$ and where $0<\gd<\min \set{\gd_1,\ldots,\gd_p}$.

\begin{theorem}
\label{thm:ORPmain}
Suppose that Assumption~\textup{\ref{ass:ORPass}} holds with
 $\kappa\ge 0$ and that $\pm i\gw_k\in \rho(A^\kappa)$ and $P_c^\kappa(\pm i\gw_k)$ is boundedly invertible for every $k\in \List{q}$.  If $a_0\neq 0$ or $c_0\neq 0$, assume also that $0\in \rho(A^\kappa)$ and that $P_c^\kappa(0)$ is boundedly invertible.
Let $\ureg$ be as in~\eqref{eq:ureg}--\eqref{eq:uregcoeff}.
The control law
\eqn{
\label{eq:ControlLaw}
u(t)= -\kappa e(t) + \ureg(t), \qquad t\geq 0
}
with $e=y-\yref$
solves the \SORP\ for those reference and disturbance signals of the form~\eqref{eq:yrefwdist}
for which $\ureg(t)\in \Ulim^\gd$ for some $\gd>0$ and for all $t\geq 0$.
For every $x_0\in X$ the state trajectory $x$ is uniformly bounded and
the following hold.
\begin{itemize}
\item If $\kappa>0$, then 
there exists a set $\Omega\subset \zinf$ of finite measure such that 
 $e\in \Lp[2](\zinf\setminus\Omega;U)$ and $ e\in \Lp[1](\Omega;U)$.
\item If $\kappa = 0$, then $\norm{e}_{\Lp[2](t,t+1)}\to0 $ as $t\to\infty$. If, in addition, $\T$ is exponentially stable, then $t\mapsto e^{\ga t}e(t)\in \Lp[2](0,\infty;U)$ for some $\ga>0$.
\end{itemize}
If $x_0\in X$, $\wdist$, and $\yref$ are such that
 there exists an element $v\in U$ satisfying
 $(x_0,\phi(\ureg(0)-\kappa (v- \yref(0))),\wdist(0))^\top \in \Dom(\Stot)$ and $v=\CD(x_0,\phi(\ureg(0)-\kappa (v- \yref(0))),\wdist(0))^\top$, then 
$(x,\ureg+\kappa \yref,\wdist,y)$ is a classical solution of~\eqref{eq:SysGenInputs} and $e\in \Hloc{1}(0,\infty;U)$.

If 
 $\CD\in \Lin(\Dom(S),U)$ extends to a bounded operator $[C,D_c,D_d]\in \Lin(X\times U\times U_d,U)$
where either $D_c=0$ or $\re D_c\ge cI$ for some $c>0$,
then for all $x_0\in X$ we have $\norm{y(t)-\yref(t)}\to 0$ as $t\to\infty$.
\end{theorem}

\begin{remark}
\label{eq:Gammavcond}
If $e(t)\to0$ as $t\to\infty$, then
$ \ureg$ is the asymptotic part of the control signal $u$ and the condition that $ \ureg(t) \in \Ulim^\gd$ for all $t\ge 0$ in Theorem~\ref{thm:ORPmain} has the interpretation that 
 $u(t)$ belongs to the linear regime of $\phi$ for large $t\ge 0$.
If $\norm{r_k}<\gd_k$ for all $k\in \List{p}$, then 
the set $\Lambda$ of coefficients $(a_k,b_k,c_k,d_k)_k$ for which $\ureg(t)\in\Ulim^\gd$ for all $t\geq 0$ contains a ball centered at $0$. 
\end{remark}

\begin{remark}
\label{rem:FFvsFB}
The control law~\eqref{eq:ControlLaw} in \cref{thm:ORPmain} incorporates feedforward control, and it utilises knowledge of the signals $\yref$ and $\wdist$, as well as the transfer function values $P_c^\kappa(\pm i\gw_k)$ of the system.
In the linear case (in the absence of input saturation), output regulation can also be achieved without information on the amplitudes and phases of $\yref$ and $\wdist$ by 
using an internal model based dynamic error feedback controller~\cite{RebWei03,Pau16a}.
 Internal model based control for finite-dimensional systems with input saturation and time-varying $\yref$ and $\wdist$ have been considered
 in~\citel{SabSto00book}{Sec.~3.5} and~\cite{LeiWan25}, and
developing similar controller designs for infinite-dimensional systems is an interesting topic for further research.
\end{remark}

The proof of \cref{thm:ORPmain} utilises the following result which guarantees the existence of classical and generalised solutions of the system~\eqref{eq:SysGenInputs}. 
Here we denote $\Phi = [\Phi^c,\Phi^d]$ and $\F = [\F^c,\F^d]$.
The proofs of \cref{thm:ORPmain} and \cref{prp:OpenLoopSol} are presented in Section~\ref{sec:Proofs}.

\begin{proposition}
\label{prp:OpenLoopSol}
Suppose  Assumption~\textup{\ref{ass:ORPass}(a)} holds and $\kappa\ge 0$. If $x_0\in X$,  $u_1\in \Lploc[2](0,\infty;U)$, and $\wdist\in \Lploc[2](0,\infty;U_d)$, then~\eqref{eq:SysGenInputs} 
has a unique generalised solution  $(x,u_1,\wdist,y)$ satisfying $x(0)=x_0$. This solution satisfies%
\begin{subequations}
\label{eq:OpenLoopStateandOutput}
\eqn{
x(t) &=  \T_t x_0 + \Phi_t^c \Pt\phi(u_1-\kappa y) + \Phi_t^d\Pt \wdist, \qquad t\geq0\\ 
y &= \Psi_\infty x_0 + \F_\infty^c \phi(u_1-\kappa y) + \F_\infty^d \wdist.
}
\end{subequations}
If $x_0\in X$, $u_1\in \Hloc{1}(0,\infty;U)$ and $\wdist\in \Hloc{1}(0,\infty;U)$ and if there exists $v\in U$ satisfying
 $(x_0,\phi(u_1(0)-\kappa v),\wdist(0))^\top \in \Dom(\Stot)$ and $v=\CD(x_0,\phi(u_1(0)-\kappa v),\wdist(0))^\top$, then 
 $(x,u_1,\wdist,y)$ is a classical solution of~\eqref{eq:SysGenInputs}
and $y\in \Hloc{1}(0,\infty;U)$.
\end{proposition}

\subsection{Computing The Control Law}
\label{sec:ComputingTheControl}
In this section we comment on the computation of the parameters~\eqref{eq:uregcoeff} of the function $\ureg$ in~\eqref{eq:ureg}.
We can first note that if $\pm i\gw_k\in \rho(A)$ for some $k\in \List[0]{q}$, then  the expressions
$ P_c^\kappa(\pm i\gw_k)\inv  (I-\kappa P_c^\kappa(\pm i\gw_k)) $
 and $P_c^\kappa(\pm i\gw_k)\inv P_d^\kappa(\pm i\gw_k)$ in the formulas~\eqref{eq:uregcoeff} simplify to 
$ P_c(\pm i\gw_k)\inv $ and $P_c(\pm i\gw_k)\inv P_d(\pm i\gw_k)$, respectively.

If $u\in U$ and $w\in U_d$, then 
$y=P_c^\kappa(\pm i\gw_k)u+P_d^\kappa(\pm i\gw_k)w$ when $x\in X$ is such that $(x,u-\kappa y,w)^\top\in \Dom(S)$ and 
\eqn{
\label{eq:SysLapTrans}
\pmat{\pm i\gw_k x\\ y} = \Stot \pmat{x\\u-\kappa y\\ w}.
}
Because of this, 
the transfer function values in~\eqref{eq:uregcoeff} can be computed by solving $x$ and $y$ from~\eqref{eq:SysLapTrans}. The system~\eqref{eq:SysLapTrans} typically corresponds to taking a formal Laplace transform of the original PDE system under negative output feedback $u(t)=-\kappa y(t)+\tilde u(t)$ and evaluating it at $\gl=\pm i\gw_k$~\citel{JacZwa12book}{Ch.~12},~\citel{PauHumCDC22}{Sec.~4}. For simple PDE models $x$ and $y$ can be constructed explicitly, and for more complicated ones they can be approximated numerically.
Small numerical errors lead to approximate tracking with a relatively small error.

In the situation where the amplitudes $(c_k)_k$ and $(d_k)_k$ of $\wdist$ and the disturbance transfer function $P_d$ are unknown, the quantities $P_d^\kappa(0) c_0$ and $P_d^\kappa (\pm i\gw_k)(c_k\mp id_k)$ can often be measured directly from the asymptotic response of the system~\eqref{eq:SysMain} with the control input $u(t)=-\kappa y(t)$. Indeed, if the assumptions of \cref{thm:ORPmain} hold with $\kappa\ge 0$
 we can 
define a reference signal $\yref(t)$ of the form~\eqref{eq:yrefwdist} with amplitudes 
\eqn{
\label{eq:DistMeasurementFormulas}
a_0 = P_d^\kappa(0)c_0, \qquad \mbox{and} \qquad a_k\mp ib_k = P_d^\kappa(\pm i\gw_k)(c_k\mp id_k).
}
Then it is straightforward to check that the corresponding $\ureg(t)$ defined by~\eqref{eq:ureg} and~\eqref{eq:uregcoeff} satisfies $\ureg(t)+\kappa \yref(t)\equiv 0$.
Because of this, \cref{thm:ORPmain} implies that
if $\ureg(t)\in \cramped{\Ulim^\gd}$, $t\ge 0$, for some $\gd>0$, then the output $y$ converges to the reference $\yref$ (in the sense of \cref{thm:ORPmain}).
This allows us to identify the amplitudes $(a_k)_k$ and $(b_k)_k$ from the asymptotic output $y$, and the quantities $P_d^\kappa(0) c_0$ and $P_d^\kappa (\pm i\gw_k)(c_k\mp id_k)$ are directly determined by the formulas~\eqref{eq:DistMeasurementFormulas}.
Even though $c_0$, $d_0$, and $P_d$ are unknown, 
we can verify the condition $\ureg(t)\in \cramped{\Ulim^\gd}$, $t\ge 0$, required in this process 
by checking that the asymptotic part of the control input $u(t)=-\kappa y(t)$ remains within the linear regime of the saturation function $\phi$.

\subsection{Proofs of the Main Results}
\label{sec:Proofs}

In the proof of Theorem~\ref{thm:ORPmain} we make use of the fact that the 
 reference and disturbance signals in~\eqref{eq:yrefwdist} can be expressed as outputs of a finite-dimensional \emph{exosystem} of the form
\begin{subequations}
\label{eq:exo}
\eqn{
\dot v(t) &= \Aexo v(t), \qquad v(0)=v_0\in \C^{2q+1}\\
\wdist(t) &= Ev(t)\\
\yref(t) &= Fv(t),
}
\end{subequations}
where $\Aexo\in\C^{(2q+1)\times (2q+1)} $ is skew-Hermitian, i.e., $\Aexo^\ast = -\Aexo$, and 
 $E\in \Lin(\C^{2q+1},$ $U_d)$ and $F\in \Lin(\C^{2q+1},U)$.
More precisely, we define
\begin{subequations}
\label{eq:exoparams}
\eqn{
\Aexo = \blkdiag \left( 0, \pmat{0&\gw_1\\-\gw_1& 0}, \ldots, \pmat{0&\gw_q\\-\gw_q& 0}\right)
}
and define $E$ and $F$ in such a way that 
\eqn{
E v &= c_0 v_0 + c_1v_1^1-d_1v_1^2 + \cdots +  c_qv_q^1-d_qv_q^2 \\
F v &= a_0 v_0 + a_1v_1^1-b_1v_1^2 + \cdots +  a_qv_q^1-b_qv_q^2 
}
\end{subequations}
for $v = (v_0,v_1^1,v_1^2,\ldots,v_q^1,v_q^2)^\top$.
Then $\yref$ and $\wdist$ are exactly the outputs of~\eqref{eq:exo} corresponding to 
 the initial state $v_0=(1,1,0,1,0,\ldots,1,0)^\top$.

\begin{remark}
\label{rem:ZeroBlock}
If $a_0=0$ and $c_0=0$ in~\eqref{eq:yrefwdist}, then the first blocks of $\Aexo$, $E$, and $F$ corresponding to the frequency $\gw_0=0$ can be removed. In this situation $P_c^\kappa(0)$ does not need  to be boundedly invertible.
These blocks
are included in $\Aexo$, $E$ and $F$ throughout this section, but the adaptation of the statements and proofs to the case where $a_0=0$ and $c_0=0$ is trivial.
\end{remark}

Throughout this section
 $\Sigma = (\T,[\Phi^c,\Phi^d],\Psi,[\F^c,\F^d])$ is the well-posed system associated with the system node $\Stot$ in~\eqref{eq:SysMain}. 
We denote by
\eq{
S_c=\pmat{\AB_c\\ \CD_c}
\qquad
\mbox{and}
\qquad
S_d=\pmat{\AB_d\\ \CD_d}
}
 the system nodes of the well-posed systems $(\T,\Phi^c,\Psi,\F^c)$ and $(\T,\Phi^d,\Psi,\F^d)$, respectively.
The transfer function of $S$ has the form $P=[P_c,P_d]$, where $P_c$ and $P_d$ are the transfer functions of $S_c$ and $S_d$, respectively.
Assumption~\ref{ass:ORPass}(a) and~\citel{Sta02}{Thm.~4.2} imply that $(\T,\Phi^c,\Psi,\F^c)$ is impedance passive and that $\re P_c(\gl)\geq c_\gl I$ for some $\gl,c_\gl>0$.
Moreover, we continue to denote by $P^\kappa = [P_c^\kappa,P_d^\kappa]$ the transfer function of the closed-loop system
$(\T^\kappa,\Phi^\kappa,\Psi^\kappa,\F^\kappa)$
obtained from $(\T,\Phi,\Psi,\F)$ with admissible output feedback
$\pmatsmall{-\kappa I\\0}$. We recall that $A^\kappa$ in \cref{ass:ORPass}(b') is the generator of $\T^\kappa$.
Finally, we denote the system node associated to
$(\T^\kappa,\Phi^\kappa,\Psi^\kappa,\F^\kappa)$
by  $S^\kappa = \pmatsmall{A^\kappa\&B^\kappa\\ C^\kappa\& D^\kappa}$, and denote $B^\kappa=[B_c^\kappa,B_d^\kappa]$.

The solvability of the ``regulator equations''~\eqref{eq:RegEqns} in the following proposition is both necessary and sufficient for the solvability of the \emph{linear} output regulation problem without the input saturation~\cite{ByrLau00,NatGil14}.
Our solvability result generalises corresponding results in~\citel{ByrLau00}{Sec.~V} and~\citel{NatGil14}{Sec.~V}.

\begin{proposition}
\label{prp:RegEqns}
Suppose that Assumption~\textup{\ref{ass:ORPass}(a)} is satisfied. 
Assume that $\Aexo\in \C^{q'\times q'}$ is skew-Hermitian, $\gs(\Aexo)\cap \gs(A^\kappa)=\varnothing$, and $P_c^\kappa(\gl)$ is boundedly invertible for every $\gl\in\gs(\Aexo)$. Then
the \emph{regulator equations} 
\eqn{
\label{eq:RegEqns}
\pmat{\Pi \Aexo\\ F} = 
S
 \pmat{\Pi\\ \Gamma \\ E}
} 
have a unique solution $(\Pi,\Gamma)^\top \in \Lin(\C^{q'},X\times U)$ such that $\ran( (\Pi,\Gamma,E)^\top)\subset \Dom(\Stot)$.
The solution is defined by the formulas
\eq{
\Pi \varphi_k &=  (\gl_k-A^\kappa)\inv B_c^\kappa(\Gamma \varphi_k + \kappa F\varphi_k) + (\gl_k-A^\kappa)\inv B_d^\kappa E\varphi_k\\
 \Gamma \varphi_k&=  P_c^\kappa(\gl_k)\inv ((I-\kappa P_c^\kappa(\gl_k)) F\varphi_k - P_d^\kappa(\gl_k)E\varphi_k),
}
where $\set{\varphi_k}_{k=1}^{q'}$ are the linearly independent eigenvectors of $\Aexo$ corresponding to the eigenvalues $\set{\gl_k}_{k=1}^{q'}$ so that $\Aexo \varphi_k = \gl_k\varphi_k$ for $k\in \List{q'}$.
\end{proposition}

\begin{proof}
We have from~\citel{Sta05book}{Cor.~7.4.11, Def.~7.4.2} that 
$I + \pmatsmall{0\\\kappa \CD \\0}$ maps $\Dom(S)$ continuously one-to-one onto $\Dom(S^\kappa)$ and that
$S=S^\kappa \left( I + \pmatsmall{0\\\kappa \CD \\0} \right)$.
We can therefore rewrite the regulator equations~\eqref{eq:RegEqns} as
\eq{
\pmat{\Pi \Aexo\\ F} = 
S \pmat{\Pi\\ \Gamma \\ E}
=S^\kappa \left( I + \pmat{0\\\kappa \CD \\0} \right)\pmat{\Pi\\ \Gamma \\ E}
=S^\kappa \pmat{\Pi\\ \Gamma + \kappa F \\ E}.
}
Since the eigenvectors $\set{\varphi_k}_k$ form a basis of $\C^{q'}$,
an operator $(\Pi,\Gamma)^\top \in \Lin(\C^{q'},X\times U)$ satisfying $\ran( (\Pi,\Gamma,E)^\top)\subset \Dom(\Stot)$ satisfies 
 the regulator equations  if and only if 
\eqn{
\label{eq:RegEqndecomp}
\begin{cases}
(\gl_k -A^\kappa)\Pi \varphi_k = B_c^\kappa(\Gamma \varphi_k + \kappa F\varphi_k) + B_d^\kappa E\varphi_k\\
F\varphi_k = C^\kappa \&D^\kappa \pmat{\Pi\varphi_k \\ \Gamma\varphi_k + \kappa F\varphi_k \\ E\varphi_k}
\end{cases}
}
for all $k\in \List{q'}$.
Since $\gl_k\in \rho(A^\kappa)$ by assumption, the first equation is equivalent to $\Pi \varphi_k = R_k^\kappa B_c^\kappa (\Gamma \varphi_k+\kappa F\varphi_k) + R_k^\kappa B_d^\kappa  E\varphi_k$, where 
 $R_k^\kappa =(\gl_k-A^\kappa )\inv$.
Using the properties of the system node $S^\kappa$ the second line becomes
$F\varphi_k
= P_c^\kappa(\gl_k)(\Gamma \varphi_k+\kappa F \varphi_k) + P_d^\kappa(\gl_k)E\varphi_k$.
Since $P_c^\kappa(\gl_k)$ is boundedly invertible for all $k\in \List{q'}$ by assumption, the above equation 
is equivalent to
$ \Gamma \varphi_k = P_c^\kappa(\gl_k)\inv ( (I-\kappa P_c^\kappa(\gl_k))F\varphi_k - P_d(\gl_k)E\varphi_k)  $
for every $k$.
These arguments confirm that if
$(\Pi,\Gamma)^\top \in \Lin(\C^{q'},X\times U)$ such that $\ran( (\Pi,\Gamma,E)^\top)\subset \Dom(\Stot)$ satisfies the regulator equations, then it is necessarily of the form given in the claim. 
On the other hand, if $\Pi$ and $\Gamma$ are defined by the formulas in the claim, then the above arguments together with~\citel{Sta05book}{Def.~7.4.2} imply that $\ran( (\Pi,\Gamma,E)^\top)\subset \Dom(\Stot)$ and that the regulator equations~\eqref{eq:RegEqns} are satisfied.
\end{proof}

\begin{lemma}
\label{lem:uregForm}
Let
Assumption~\textup{\ref{ass:ORPass}(a)} hold, let
 $\Aexo$, $E$ and $F$ be as in~\eqref{eq:exoparams}, and assume that $\set{\pm i\gw_k}_k\subset \rho(A^\kappa)$ and $P_c^\kappa(\pm i\gw_k)$ is boundedly invertible for all $k$.
Moreover, let $(\Pi,\Gamma)^\top$ be the unique solution of the regulator equations~\eqref{eq:RegEqns} and let $v(t)=e^{\Aexo t}v_0$, $t\ge 0$, with $v_0=(1,1,0,1,0,\ldots,1,0)^\top$. Then
$\Gamma v(t)=\ureg(t)$, $t\ge 0$, where
$\ureg(t)$ is as in~\eqref{eq:ureg} and~\eqref{eq:uregcoeff}.
\end{lemma}

\begin{proof}
We have 
$v(t)
= (1,\cos(\gw_1 t),-\sin(\gw_1 t), \ldots, \cos(\gw_q t),-\sin(\gw_q t))^\top$.
If $\{\varphi_0,$ $\varphi_1^\pm,\ldots,\varphi_q^\pm\}$ are
the normalised eigenvectors of $\Aexo$ corresponding to the eigenvalues $\set{0,\pm i\gw_1,\ldots,\pm i\gw_q}$, then
\eqn{
\label{eq:Gammavgenformula}
\Gamma v(t)
= \ga_0(t)\Gamma \varphi_0 
+ \sum_{k=1}^q \ga_k(t) \Gamma \varphi_k^+ + \gb_k(t) \Gamma \varphi_k^- ,
}
where  $\set{\ga_k(t)}_k$ and $\set{\gb_k(t)}_k$  are such that $v(t)=\ga_0(t)\varphi_0+\sum_{k=1}^q(\ga_k(t)\varphi_k^+ +\gb_k(t)\varphi_k^-)$. 
Since $\varphi_0=(1,0,\ldots,0)^\top$ and $\varphi_k^\pm = 2^{-1/2}(0,\ldots,0,1,\pm i,0,\ldots,0)^\top$ for $k\in \List{q}$,
a direct computation shows that
$\ga_0(t)\equiv 1$, $\ga_k(t)=2^{-1/2}(\cos(\gw_kt) +i\sin(\gw_kt))$, and $\gb_k(t)=\conj{\ga_k(t)}$ for $t\ge 0$ and $k\in \List{q}$.
We have
 $E\varphi_0=c_0$, $F\varphi_0=a_0$, $E\varphi_k^\pm = 2^{-1/2} (c_k \mp id_k)$, and $F\varphi_k^\pm = 2^{-1/2} (a_k \mp ib_k)$ for $k\in \List{q}$.
Substituting these expressions to the formulas for $\Gamma\varphi_k^\pm$ in \cref{prp:RegEqns} and to~\eqref{eq:Gammavgenformula} completes the proof.
\end{proof}

\begin{lemma}
\label{lem:SatInnerProd}
Let $U$ be a Hilbert space and let $\phi:U\to U$ be the symmetric saturation function defined so that $\phi(u)=u$ whenever $\norm{u}\le \gd_0$ and $\phi(u)=\gd_0\norm{u}\inv u$ whenever $\norm{u}>\gd_0$ for some fixed $\gd_0>0$.
Assume that $\kappa>0$ and $0<\gd<\gd_0$.
If
$e\in U$ and $u\in U$ 
 are such that $\norm{u}\leq \gd_0-\gd$, then
\eq{
\re \iprod{e}{\phi(u-\kappa e) - u}
\leq - \frac{ \kappa\gd\norm{e}^2}{\max \set{\gd_0,\norm{u-\kappa e}}}.
}
\end{lemma}

\begin{proof}
Assume that $u,e\in U$ satisfy $\norm{u}\leq \gd_0-\gd$.
If  $\norm{u-\kappa e}\le \gd_0$, then $\phi(u-\kappa e)=u-\kappa e$ and 
$\re \iprod{e}{\phi(-\kappa e +u)-u} = -\kappa \norm{e}^2$ and $0<\gd<\gd_0$ implies that the estimate in the claim holds.
On the other hand, if $\Delta:=\norm{u-\kappa e}> \gd_0$, then 
$\phi(u-\kappa e)=\frac{\gd_0}{\Delta}(u-\kappa e)$ and
\eq{
 \Delta\re \iprod{e}{\phi(u-\kappa e )-u} 
&=\re \iprod{ e}{\gd_0(u-\kappa e) -\Delta u} \\
&=(\gd_0-\Delta)\re \iprod{ e}{ u} 
- \kappa\gd_0\norm{ e}^2.
}
If $\re \iprod{ e}{ u}\ge 0$, then the right-hand side is at most $ -\kappa\gd_0\norm{ e}^2$ and thus the estimate in the claim holds.
It remains to consider the case where $\Delta:=\norm{u-\kappa e}> \gd_0$ and 
$\re \iprod{ e}{ u}< 0$.
Then 
\eq{
\MoveEqLeft \Delta\re \iprod{e}{\phi(u-\kappa e )- u} 
=(\gd_0-\Delta)\re \iprod{ e}{ u} - \kappa\gd_0\norm{ e}^2\\
&=(\Delta-\gd_0)\abs{\re \iprod{ e}{ u}} - \kappa\gd_0\norm{ e}^2
\le (\norm{ u} + \norm{\kappa e}-\gd_0)\norm{ e}\norm{ u} - \kappa \gd_0\norm{e}^2\\
&= (\norm{ u}-\gd_0)\norm{ u}\norm{ e} +(\norm{ u} - \gd_0)\kappa \norm{e}^2
\le -\gd\kappa \norm{e}^2.
}
since $\norm{ u}-\gd_0\leq -\gd$ by assumption.
This shows that the estimate in claim holds also in the case where $\norm{u-\kappa e}> \gd_0$ and 
$\re \iprod{ e}{ u}< 0$.
\end{proof}

\begin{proof}[Proof of Proposition~\textup{\ref{prp:OpenLoopSol}}]
If $\kappa=0$, the claims follow from~\citel{TucWei14}{Prop.~4.6--4.7}. Assume now that $\kappa>0$.
Using the structure of $\phi$ and the properties in~\citel{HasPau25arxiv}{Ex.~3.7} it is straightforward to verify that 
\eqn{
\label{eq:WPphiestimate}
\re \iprod{\phi(v_2)-\phi(v_1)}{v_2-v_1}_U
&\ge \norm{\phi( v_2)-\phi( v_1)}_U^2, \qquad v_1,v_2\in U.
}
If we denote $\tilde y = \kappa^{1/2}y$, 
 $\tilde u_1 = \kappa^{-1/2}u_1$, and $\phi_\kappa = \kappa^{-1/2}\phi(\kappa^{1/2}\cdot)$
and define a system node $S_\kappa = \diag(1,\sqrt{\kappa})S \diag(1,\sqrt{\kappa},1)$ with domain
$D(S_\kappa) = \{(x,u,w)^\top\in X\times U \times U_d\; | \;  Ax + B\pmatsmall{\sqrt{\kappa}u\\ w}\in X\}$,
then $\phi_\kappa$ satisfies~\eqref{eq:WPphiestimate} and
we can reduce~\eqref{eq:SysGenInputs} to an equivalent system with $\kappa=1$.
Because of this, we may assume that $\kappa=1$
and that the function $\phi:U\to U$ satisfies~\eqref{eq:WPphiestimate}.

\hspace{-1ex}
To show the existence of classical solutions of~\eqref{eq:SysGenInputs}, let $x_0\in X$,
$u_1\in \Hloc{1}(0,\infty;U)$, and $\wdist \in \Hloc{1}(0,\infty;U_d)$ and assume $v\in U$ is such that
 $(x_0,\phi(u_1(0)- v),\wdist(0))^\top \in \Dom(\Stot)$ and $v=\CD(x_0,\phi(u_1(0)- v),\wdist(0))^\top$.
Let $\gl\in\rho(A)$ and $z\in X$, denote $R_\gl = (\gl-A)\inv$ for brevity, and
 define $x_{d0} = R_\gl z+R_\gl B_d\wdist(0)$. Then $Ax_{d0}+B_d\wdist(0)\in X$ and we have from~\citel{TucWei14}{Prop.~4.6} that if we define $x_d$ and $y_d$ by
\begin{subequations}
\label{eq:WPdpartWPS}
\eqn{
x_d(t) &=  \T_t x_{d0} +  \Phi_t^d\Pt \wdist, \qquad t\geq0\\ 
y_d &= \Psi_\infty x_{d0} +  \F_\infty^d \wdist,
}
\end{subequations}
then $x_d\in C^1(\zinf;X)$, $x_d(0)=x_{d0}$, $y_d\in \Hloc{1}(0,\infty;U)$,  and
\eqn{
\label{eq:WPdpartSysNode}
\pmat{\dot x_d(t)\\ y_d(t)} = S_d \pmat{x_d(t)\\\wdist(t)}, \qquad t\geq0.
}
Moreover, $y_d(0)=CR_\gl z + P_d(\gl)w_d(0)$.
If we define $v_c=v-y_d(0)$,
then it is straightforward to check that $A(x_0-x_{d0})+B_c\phi(u_1(0)- y_d(0)- v_c)\in X$  and
$\CD_c (x_0-x_{d0}, \phi(u_1(0)- y_d(0)- v_c))^\top=v_c$.
If we define $u_c=u_1- y_d\in \Hloc{1}(0,\infty;U)$, 
then~\eqref{eq:WPphiestimate} and~\citel{HasPau25arxiv}{Thm.~3.3} imply that there exist $x_c\in C^1(\zinf;X)$ and $y_c\in \Hloc{1}(0,\infty;U)$ satisfying $x_c(0)=x_0-x_{d0}$,
\eqn{
\label{eq:WPcpartSysNode}
\pmat{\dot x_c(t)\\ y_c(t)} = S_c \pmat{x_c(t)\\\phi(u_c(t)- y_c(t))}, \qquad t\geq0,
}
and
\begin{subequations}
\label{eq:WPcpartWPS}
\eqn{
x_c(t) &=  \T_t (x_0-x_{d0}) + \Phi_t^c \Pt\phi(u_c- y_c), \qquad t\geq 0 \\
y_c &= \Psi_\infty (x_0-x_{d0}) + \F_\infty^c \phi(u_c- y_c) .
}
\end{subequations}
The definitions and properties of $x_c$, $x_d$, $y_c$, and $y_d$ imply that  $x=x_c+x_d$ and $y=y_c+y_d\in \Hloc{1}(0,\infty;U)$ satisfy~\eqref{eq:OpenLoopStateandOutput} and $x(0)=x_0$ and that $(x,u_1,\wdist,y)$ is a classical solution of~\eqref{eq:SysGenInputs}.

To analyse generalised solutions, let 
$x_0\in X$,  $u_1\in \Lploc[2](0,\infty;U)$, and $\wdist\in \Lploc[2](0,\infty;U_d)$.
Let $x_{d0}=0$ and define $x_d$ and $y_d$ by~\eqref{eq:WPdpartWPS}.
 We have from~\citel{TucWei14}{Prop.~4.7} that $x_d\in C(\zinf;X)$ and $y_d\in \Lploc[2](0,\infty;U)$
and that
 there exist $(x_d^k)_{n\in\N}\subset C^1(\zinf;X)$, $(\wdist^k)_{k\in\N}\subset C(\zinf;U_d)$ and $(y_d^k)_{k\in\N}\subset C(\zinf;U)$ such that $(x_d^k,\wdist^k,y_d^k)$ 
satisfy~\eqref{eq:WPdpartSysNode} for $k\in\N$
 and 
$(\Pt[\tau]x_d^k,\Pt[\tau]\wdist^k,\Pt[\tau]y_d^k)^\top\to (\Pt[\tau]x_d,\Pt[\tau]\wdist,$ $\Pt[\tau]y_d)^\top
$
as $k\to \infty$ for any $\tau>0$.
Define $u_c=u_1- y_d\in \Lploc[2](0,\infty;U)$. We have from~\citel{HasPau25arxiv}{Thm.~3.3} that there exist 
 $x_c\in C(\zinf;X)$ and $y_c\in \Lploc[2](0,\infty;U)$ satisfying $x_c(0)=x_0$ and~\eqref{eq:WPcpartWPS} and that there exist 
$(x_c^k)_{k\in\N}\subset C^1(\zinf;X)$, $(u_c^k)_{k\in\N}\subset C(\zinf;U)$ and $(y_c^k)_{k\in\N}\subset C(\zinf;U)$ such that $(x_c^k,u_c^k,y_c^k)$ 
satisfy~\eqref{eq:WPcpartSysNode} for $k\in\N$
 and   
 $(\Pt[\tau]x_c^k,\Pt[\tau]u_c^k,\Pt[\tau]y_c^k)^\top\to (\Pt[\tau]x_c,\Pt[\tau]u_c,\Pt[\tau]y_c)^\top$
 as $k\to \infty$ for every $\tau>0$.
These properties imply that if we define
 $x=x_c+x_d\in C(\zinf;X)$, $y=y_c+y_d\in \Lploc[2](0,\infty;U)$,
 $x^k=x_c^k+x_d^k\in C^1(\zinf;X)$, $y^k=y_c^k+y_d^k\in \Lploc[2](0,\infty;U)$ and $u_1^k = u_c^k+ y_d^k$ for $k\in\N$, then $x$ and $y$ satisfy~\eqref{eq:OpenLoopStateandOutput}, and the sequence $(x^k,u_1^k,\wdist^k,y^k)_{k=1}^\infty$
has the properties in \cref{def:SysNodeStates}.
Thus $(x,u_1,\wdist,y)$ is a generalised solution of~\eqref{eq:SysGenInputs}.

To prove the uniqueness of the generalised solution, assume first that 
$u_1^1,u_1^2\in C(\zinf;U)$ and $\wdist^1,\wdist^2\in C(\zinf;U_d)$, and that
$x_1$, $x_2$, $y_1$, and $y_2$ are such that 
$(x_1,u_1^1,\wdist^1,y_1)$ and $(x_2,u_1^2,\wdist^2,y_2)$ are two classical solutions of~\eqref{eq:SysGenInputs}. 
If we define $z_k(t)=\Phi_t^d \Pt \wdist^k$, $t\geq 0$, $y_d^k=\F_\infty^d \wdist^k$, for $k=1,2$, and
denote $x=x_2-x_1$, $y=y_2-y_1$, $z=z_2-z_1$, $y_d=y_d^2-y_d^1$,
 and $\phi_k=\phi(u_1^k- y_k)$, $k=1,2$, then 
\eq{
x(t)-z(t) &=  \T_t x(0) + \Phi_t^c \Pt(\phi_2-\phi_1) , \qquad t\geq0\\ 
y-y_d &= \Psi_\infty x(0) + \F_\infty^c (\phi_2-\phi_1) .
}
The impedance passivity of $\Sigma_c = (\T,\Phi^c,\Psi,\F^c)$, \citel{HasPau25arxiv}{Lem.~3.12(b)}, $z(0)=0$, and  
 $\norm{\Pt(y_d^2-y_d^1)} 
\le \norm{\F_t^d} \norm{\wdist^2-\wdist^1}_{\Lp[2](0,t)}$
imply that
\eq{
\MoveEqLeft[8]\norm{x(t)-z(t)}^2 - \norm{x(0)}^2 
+ \norm{\phi_2-\phi_1}_{\Lp[2](0,t)}^2\\
&\le 
 2\norm{u_1^2- u_1^1}_{\Lp[2](0,t)}^2 + 2 \norm{\F_t^d}^2\norm{\wdist^2- \wdist^1}_{\Lp[2](0,t)}^2 
}
Since 
 $\norm{y}_{\Lp[2](0,t)}\le \norm{\Psi_t}\norm{x(0)} + \norm{\F_t^c}\norm{\phi_2-\phi_1}_{\Lp[2](0,t)}+\norm{y_d}_{\Lp[2](0,t)}$
and
$\norm{z(t)}
\leq \norm{\Phi_t^d}\norm{\wdist^2-\wdist^1}_{\Lp[2](0,t)}$,  for every $t>0$ there exists $M_t>0$ such that 
\eq{
\MoveEqLeft\norm{x(t)} + \norm{y}_{\Lp[2](0,t)}\\
&\le
M_t\left( \norm{x_2(0)-x_1(0)}  +  \norm{u_1^2- u_1^1}_{\Lp[2](0,t)} + \norm{\wdist^2- \wdist^1}_{\Lp[2](0,t)}\right).
}
Thus in the classical solutions of~\eqref{eq:SysGenInputs}, the functions $x\in C^1(\zinf;X)$ and $y\in C(\zinf;U)$ are uniquely determined by $x(0)$, $u_1\in C(\zinf;U)$, and $\wdist\in C(\zinf;$ $U_d)$.
Moreover, it is straightforward to use the the last estimate above to show that also the generalised solutions of~\eqref{eq:SysGenInputs} are uniquely determined by $x(0)\in X$, $u_1\in \Lploc[2](0,\infty;U)$, and $\wdist\in\Lploc[2](0,\infty;U_d)$.
\end{proof}

\begin{proof}[Proof of Theorem~\textup{\ref{thm:ORPmain}}]
Let $x_0\in X$ and let $\yref$ and $\wdist$ be of the form~\eqref{eq:yrefwdist} for some fixed coefficients $(a_k)_k,(b_k)_k\subset U$, $(c_k)_k,(d_k)_k\subset U_d$ such that $\ureg(t)\in \Ulim^\gd$ for some $\gd>0$ and for all $t\ge 0$.
The control input has the form $u =  \ureg+\kappa \yref -\kappa y $, where $ \ureg+\kappa \yref  \in C^\infty(\zinf;U)$.  \cref{prp:OpenLoopSol} implies that the controlled system~\eqref{eq:SysGenInputs} has a unique generalised solution $(x,u_1,\wdist,y)$ with $u_1= \ureg+\kappa \yref $ in the sense of \cref{def:SysNodeStates}. Moreover, this solution satisfies~\eqref{eq:OpenLoopStateandOutput}.
Finally, if there exists $v\in U$ such that $(x_0,\phi(u_1(0)-\kappa v),\wdist(0))^\top \in \Dom(S)$
and $v=\CD(x_0,\phi(u_1(0)-\kappa v),\wdist(0))^\top$, then 
$(x, \ureg+\kappa \yref ,\wdist,y)$ is a classical solution of~\eqref{eq:SysGenInputs} in the sense of \cref{def:SysNodeStates} and $y\in \Hloc{1}(0,\infty;U)$.

In the next part of the proof we will show that $\norm{x(\cdot)}$ is bounded and
analyse $e=y-\yref$.
Let $(\Pi,\Gamma)^\top$ be the unique solution of the regulator equations~\eqref{eq:RegEqns} in \cref{prp:RegEqns} and
define $v(t)=e^{\Aexo t}v_0$, $t\ge 0$,  with $v_0=(1,1,0,1,0,\ldots,1,0)^\top$. 
Then  $\wdist(t) =Ev(t)$ and $\yref (t) = Fv(t)$ by~\eqref{eq:exoparams}.
We have $t\mapsto \Pi v(t)\in C^\infty(\zinf;X)$ and $(\Pi v(t),\Gamma v(t),Ev(t))^\top \in \Dom(S)$, $t\geq 0$, and the regulator equations~\eqref{eq:RegEqns} and \cref{lem:uregForm} imply 
\eq{
\pmat{\ddb{t} \Pi v(t)\\ \yref(t)} 
=
\pmat{\Pi \Aexo v(t)\\ Fv(t)} = 
S \pmat{\Pi v(t)\\ \Gamma v(t)\\ Ev(t)}
=S \pmat{\Pi v(t)\\ \ureg(t)\\ \wdist(t)}, \qquad t\geq 0.
}
We have from~\citel{TucWei14}{Prop.~4.6} that 
\eq{
\Pi v(t) &=  \T_t \Pi v_0 + \Phi_t^c \Pt\ureg + \Phi_t^d \Pt\wdist , \qquad t\geq0\\ 
\yref &= \Psi_\infty \Pi v_0 + \F_\infty^c \ureg + \F_\infty^d\wdist.
}
If we define $z(t)= x(t)-\Pi v(t)$, then~\eqref{eq:OpenLoopStateandOutput} with $u_1=\ureg + \kappa \yref$ imply that%
\begin{subequations}
\label{eq:zesysstates}
\eqn{
z(t) &=  \T_t z(0) + \Phi_t^c \Pt\bigl(\phi(\ureg-\kappa e) - \ureg\bigr) , \qquad t\geq0\\ 
\label{eq:zesysstateserror}
e &= \Psi_\infty z(0) + \F_\infty^c \bigl(\phi(\ureg-\kappa e) - \ureg\bigr) .
}
\end{subequations}
\cref{ass:ORPass}(a) and~\citel{Sta02}{Thm.~4.2} imply that
the well-posed system $(\T,\Phi^c,\Psi,$ $\F^c)$ is impedance passive and thus for all $t\ge 0$
\eqn{
\label{eq:ORPPassEstimate}
\norm{z(t)}^2 - \norm{z(0)}^2 \leq 2\re  \int_0^t \iprod{e(s)}{\phi(\ureg(s)-\kappa e(s) )-\ureg(s)}ds.
}

We will now analyse $x$ and $e=y-\yref$ when $\kappa=0$.
In this case
the assumption that $\ureg(t)\in \Ulim^\gd$ for $t\ge 0$ implies 
$\phi(\ureg(t)-\kappa e(t) )-\ureg(t)\equiv 0$. Thus we have from~\eqref{eq:zesysstates} that $z(t)\to 0$ as $t\to\infty$ and $e=\Psi_\infty z(0)$.
In particular, $\norm{z(t)}$ and $\norm{x(t)}=\norm{z(t)+\Pi v(t)}$ are uniformly bounded with respect to $t\ge 0$.
If $\T$ is exponentially stable, we have $e^{\ga\cdot}e\in \Lp[2](0,\infty;U)$ for some $\ga>0$ by~\citel{Sta05book}{Thm.~2.5.7}.
On the other hand, if $\T$ is strongly stable then~\citel{TucWei14}{Def.~3.1(iii)} implies that $\norm{e}_{\Lp[2](t,t+1)}^2  = \norm{\Psi_1 \T_tz(0)}_{\Lp[2](0,1)}\le \norm{\Psi_1} \norm{\T_tz(0)}\to 0$ as $t\to\infty$.

We can now consider the case where $\kappa>0$.
If we denote $\ureg(t)=(\ureg^1(t),\ldots,$ $\ureg^p(t))^\top$ and $e(t)=(e^1(t),\ldots,e^p(t))^\top$, then
$\norm{\ureg^k(t)}\le \gd_k-\gd$ for all $t\ge 0$ and $k\in \List{p}$.
If
for $k\in\List{p}$ 
 we let $\phi_k^0: U_k\to U_k$ be the symmetric saturation function defined by $\phi_k^0(u) = (\gd_k/\max\set{\gd_k,\norm{u}})u$, $u\in U_k$, 
then $\phi_k(u)=\rk+\phi_k^0(u-\rk)$ and
 \cref{lem:SatInnerProd} implies that
\eq{
\MoveEqLeft\re \iprod{e^k(t)}{\phi_k(\ureg^k(t)-\kappa e^k(t) )-\ureg^k(t)} \\
&=\re \iprod{e^k(t)}{\phi_k^0(\ureg^k(t)-\rk-\kappa e^k(t) )-(\ureg^k(t)-\rk)} \\
&\leq - \frac{ \kappa\gd\norm{e^k(t)}^2}{\max \set{\gd_k,\norm{\ureg^k(t)-\rk-\kappa e^k(t)}}}
}
for a.e. $t\ge0$.
In particular, $\re \iprod{e(t)}{\phi(\ureg(t)-\kappa e(t) )-\ureg(t)}\leq 0$ for a.e. $t\geq 0$,  and~\eqref{eq:ORPPassEstimate} implies that $\norm{z(t)}\le \norm{z(0)}$. Thus $\norm{z(t)}$ and $\norm{x(t)}=\norm{z(t)+\Pi v(t)}$ are uniformly bounded with respect to $t\ge 0$.
The estimate~\eqref{eq:ORPPassEstimate} also implies that
\eq{
 \MoveEqLeft\sum_{k=1}^p \int_0^t \abs{\re \iprod{e^k(s)}{\phi(\ureg^k(s)-\kappa e^k(s) )-\ureg^k(s)}}ds\\
&= -\re\int_0^t  \iprod{e(s)}{\phi(\ureg(s)-\kappa e(s) )-\ureg(s)}ds
\le \frac{\norm{z(0)}^2}{2}.
}
Thus 
$ \re \iprod{e^k}{\phi(\ureg^k-\kappa e^k )-\ureg^k}\in \Lp[1](0,\infty)$
for all $k\in\List{p}$.

For a fixed representative of the equivalence class $e\in \Lploc[2](0,\infty;U)$ define $h_k=\norm{\ureg^k(\cdot)-\rk-\kappa e^k(\cdot)}$, $k\in \List{p}$, and define $\Omega_k=h_k\inv(\zinf[\gd_k])\subset \zinf$. Then for every $k\in\List{q}$ the set $\Omega_k$ is measurable, $\norm{\ureg^k(t)-\rk-\kappa e^k(t)}\ge\gd_k$ for a.e. $t\in \Omega_k$, and $\norm{\ureg^k(t)-\rk-\kappa e^k(t)}< \gd_k$ for a.e. $t\in \zinf\setminus\Omega_k$.
The property that $\norm{\ureg^k(t)-\rk}\leq \gd_k-\gd$ for $t\geq 0$ implies that for $k\in \List{p}$ and for a.e. $t\in \Omega_k$ we have $\kappa \norm{e^k(t)}\geq 
\gd$, and 
\eqn{
\label{eq:ORPerror_scalar_estimates}
\frac{\kappa\norm{e^k(t)}^2}{\norm{\ureg^k(t)-\rk-\kappa e^k(t)}}
&\ge \frac{\kappa\norm{e^k(t)}^2}{\gd_k-\gd+\kappa \norm{e^k(t)}}
\ge \frac{\gd \norm{e^k(t)}}{ \gd_k}
\ge \frac{\gd^2}{\kappa \gd_k},
}
since the function $\zinf[\gd]\ni s\mapsto s/(\gd_k-\gd+s)$ is increasing.
Together with our earlier estimates this
implies that 
\eq{
\MoveEqLeft\sum_{k=1}^p
\frac{\gd^3}{\kappa \gd_k}\int_{\Omega_k} 1 dt 
\le
 \sum_{k=1}^p \int_{\Omega_k} \frac{\kappa \gd\norm{e^k(t)}^2}{\norm{\ureg^k(t)-\rk-\kappa e^k(t)}} dt \\
&\le \sum_{k=1}^p  \int_{\Omega_k}-\re\iprod{e^k(t)}{\phi(\ureg^k(t)-\kappa e^k(t) )-\ureg^k(t)}dt
\le \frac{\norm{z(0)}^2}{2 }.
}
Thus the set $\Omega = \Omega_1\cup \cdots \cup \Omega_p$ has finite measure.
Moreover, 
using the first two inequalities in~\eqref{eq:ORPerror_scalar_estimates} we similarly have
\eq{
\sum_{k=1}^p\frac{\gd^2}{\gd_k}
\int_{\Omega_k} \norm{e^k(t)} dt
\leq \sum_{k=1}^p\int_{\Omega_k} \frac{\kappa \gd\norm{e^k(t)}^2}{\norm{\ureg^k(t)-\rk-\kappa e^k(t)}}dt\le \frac{\norm{z(0)}^2}{2},
}
and thus
$e^k\in \Lp[1](\Omega_k;U_k)$ for every $k\in \List{p}$.
Since we also have $e^k\in \Lp[\infty](\Omega_j\setminus \Omega_k;U_k)$ for all $j\neq k$,
 we can further deduce that 
$e\in \Lp[1](\Omega;U)$.
Moreover, if we denote $\Omega_0 = \zinf\setminus \Omega$, then $\phi(\ureg(t)-\kappa e(t))=\ureg(t)-\kappa e(t)$ for a.e. $t\in\Omega_0$, and~\eqref{eq:ORPPassEstimate} implies
\eq{
\kappa\int_{\Omega_0} \norm{e(t)}^2 dt 
= -\re \int_{\Omega_0} \iprod{e(t)}{\phi(\ureg(t)-\kappa e(t))-\ureg(t)}dt
\le \frac{\norm{z(0)}^2}{2}.
}
This proves that the set $\Omega$ has the properties stated in the claim, namely, that $\Omega$ is of finite measure,
 $e\in \Lp[2](\zinf\setminus \Omega;U)$, and $e\in \Lp[1](\Omega;U)$.

If $x_0\in X$, $\wdist=0$, and $\yref=0$, then the control input is given by $u(t)=-\kappa y(t)$ and the generalised solution $(x,0,0,y)$   of~\eqref{eq:SysGenInputs} satisfies
\eq{
x(t) &=  \T_t x_0 + \Phi_t^c \Pt\phi(-\kappa y) , \qquad t\geq0\\ 
y &= \Psi_\infty x_0 + \F_\infty^c \phi(-\kappa y) .
}
If $\kappa=0$, then $x(t)\to 0$ as $t\to\infty$ since 
 $\T$  is strongly stable by \cref{ass:ORPass}(b).
On the other hand, if $\kappa>0$, then the semigroup generated by $A^\kappa$ in \cref{ass:ORPass}(b') is strongly stable.
Moreover, a similar argument as above using \cref{lem:SatInnerProd}
shows that if $0<\eps<\min \set{\gd_k-\norm{r_k}}$, then 
$\re \iprod{\phi_k(u)}{u}\ge 
 \eps \norm{u}^2/\max \set{\gd_k,\norm{\rk - u}} $ for all $k\in \List{q}$.
This property and the
 structure of $\phi$
 imply that there exist $\ga,\gb,\gg>0$ such that
\eq{
\re \iprod{\phi(u)}{u}\ge
\begin{cases}
\ga\norm{u}^2 & \mbox{if } \norm{u}< \gg\\
\gb & \mbox{if } \norm{u}\ge \gg.
\end{cases}
}
We therefore have from~\citel{HasPau25arxiv}{Thm.~4.3} that $\norm{x(t)}\to 0$ as $t\to\infty$.

\hspace{-2ex}
Finally, we analyse the case where
 where $\CD$ extends to an operator $[C,D_c,D_d]\in \Lin(X\times U\times U_d,U)$. Then~\eqref{eq:zesysstateserror} implies that $e(t)=Cz(t) + D_c (\phi(\ureg(t)-\kappa e(t))-\ureg(t))$ for all $t\ge 0$. Since $\phi(\ureg-\kappa e)-\ureg\in \Lp[\infty](0,\infty;U)$, since $z$ is uniformly bounded and $C\in \Lin(X,U)$, we have $e\in \Lp[\infty](0,\infty;U)$.
If we denote $\tilde u(t)=\phi(\ureg(t)-\kappa e(t) )-\ureg(t)+\kappa e(t)$, then $\tilde u\in \Lp[\infty](0,\infty;U)$ and $\tilde u(t)=0$ for a.e. $t\in \zinf\setminus \Omega$, and thus $\tilde u\in \Lp[2](0,\infty;U)$ since $\Omega$ has finite measure.
We have from~\citel{Sta02}{Cor.~6.1} that 
$-\kappa I$ is an admissible feedback operator for the well-posed system $(\T,\Phi^c,\Psi,\F^c)$, and the resulting closed-loop system
 $(\T^\kappa,\Phi^{c,\kappa},\Psi^\kappa,\F^{c,\kappa})$
 satisfies $\sup_{t\ge 0} \norm{\Phi_t^{c,\kappa}}<\infty$.
We have therefore have from~\eqref{eq:zesysstates} and~\citel{TucWei14}{Prop.~5.15} that 
\eq{
\pmat{z(t)\\
\Pt e}
 &=  \pmat{\T_t^\kappa&\Phi^{c,\kappa}_t\\ \Psi_t^\kappa& \F_t^{c,\kappa}} \pmat{z(0) \\ \Pt\bigl(\phi(\ureg-\kappa e) - \ureg + \kappa e\bigr)} .
}
Under \cref{ass:ORPass}(b) or \cref{ass:ORPass}(b') the semigroup $\T^\kappa$ generated by $A^\kappa$ is strongly stable, and the property $\tilde u\in \Lp[2](0,\infty;U)$ together with~\citel{Sta05book}{Lem. 8.1.2(iii)} 
implies that $\norm{x(t)-\Pi v(t)}=\norm{z(t)}\to 0$ as $t\to\infty$.

It remains to show that $e(t)\to 0$ as $t\to\infty$.
We have $e(t)=Cz(t) -D_c\ureg(t) + D_c \phi(\ureg(t)-\kappa e(t))$ for all $t\ge 0$.
 If $\kappa=0$ or  $D_c=0$, then clearly $e(t)=Cz(t)\to 0$ as $t\to\infty$ since $C\in \Lin(X,U)$. 
It remains to consider the case where $\re D_c\ge cI$ for some $c>0$ and $\kappa>0$.
We saw in the proof of \cref{prp:OpenLoopSol} that $\phi$ satisfies~\eqref{eq:WPphiestimate}, and is therefore a continuous and monotone function.
 Thus~\citel{HasPau25arxiv}{Lem.~3.12(a)} implies that there exists a globally Lipschitz continuous function $g:U\times U\to U$ such that 
$e(t) = \kappa\inv g(\ureg(t),\kappa Cz(t)-\kappa D_c\ureg(t))$. This implies that $e\in C(\zinf;U)$.
We will now show that the set $\Omega\subset \zinf$ can be chosen to be bounded, i.e., there exists $t_0>0$ such that $\Omega\subset [0,t_0]$.
Since $e$ is continuous, we can use the continuous representative $e$ in defining $\Omega_k$ for $k\in \List{p}$. The definition implies that $\Omega_k$ are closed sets, and therefore $\Omega =\Omega_1\cup\cdots\cup \Omega_p$ is closed as well.
Since $\Omega$ has finite measure and since $\norm{z(t)}\to 0$ as $t\to\infty$, we can choose $t_0>0$ such that $t\notin \Omega$ and $\norm{z(t)}\leq \gd /(2\kappa\norm{(I+\kappa D_c)\inv C})$ for all $t\ge t_0$ (recall that $I+\kappa D_c$ is boundedly invertible since $\re D_c\ge cI$ for some $c>0$).
We claim that $\Omega\subset [0,t_0]$. If this is not true, then the fact that $\Omega$ is closed implies that there exists $t_1>t_0$ such that $t_1\in \Omega$ and $\zabl{t_0}{t_1}\subset \zinf\setminus \Omega$. 
For all $t\in \zabl{t_0}{t_1}\subset \zinf\setminus \Omega$ we have
$e(t)=Cz(t)-\kappa D_ce(t)$, or $e(t)=(I+\kappa D_c)\inv Cz(t)$.
Thus 
 $\norm{e(t)}\le \norm{(I+\kappa D_c)\inv C}\norm{z(t)}<\gd/(2\kappa)$ 
for all $t\in \zabl{t_0}{t_1}$.
However, the definition of $\Omega$ implies that there exists $k\in \List{p}$ such that
 $\gd_k\le\norm{\ureg^k(t_1)-\rk-\kappa e^k(t_1)}\le \gd_k-\gd + \kappa \norm{e^k(t_1)}$, which implies that $\norm{e(t_1)}\ge \norm{e^k(t_1)}\ge \gd/\kappa$. This is a contradiction since $e$ is continuous and $\norm{e(t)}<\gd/(2\kappa)$ for $t\in \zabl{t_0}{t_1}$.
Thus $\Omega \subset [0,t_0]$.
Therefore for all $t\geq t_0$ we have $e(t)=(I+\kappa D_c)\inv Cz(t)\to 0$ as $t\to\infty$.
\end{proof}

\section{Boundary Control Systems}
\label{sec:BCS}

In this section we solve the \SORP\ for an abstract boundary control system~\cite{Sal87a,MalSta06,JacZwa12book} of the form
\begin{subequations}
\label{eq:BCS}
\eqn{
\dot{x}(t)&= \AA x(t) + B_{d,1} \wdistk{1}(t),  \qquad x(0)=x_0\in X
\\
\BB_c x(t) &=  \phi(u(t)) + \wdistk{2}(t)\\
\BB_d x(t) &=  \wdistk{3}(t)\\
y(t) &= \CC x(t)
}
\end{subequations}
for $t\geq 0$.
Here $\phi $ is as in Section~\ref{sec:ORPresults} and we define $\wdist=(\wdistk{1},\wdistk{2},\wdistk{3})^\top : \zinf\to U_d$ with $U_d=U_{d,1}\times U\times U_{d,3}$.
We assume that $X$, $U$, $U_{d,1}$, and $U_{d,3}$ are Hilbert spaces, $B_{d,1}\in \Lin(U_{d,1},X)$,
and assume that $(\BB,\AA,\CC)$ with $\BB=\pmatsmall{\BB_c\\\BB_d}$ is a \emph{boundary node} on the spaces $(U\times U_{d,3},X,U)$ in the sense that 
$\AA: \Dom(\AA)\subset X\to X$, $\BB\in \Lin(\Dom(\AA),U\times U_{d,3})$,  $\CC\in \Lin(\Dom(\AA),U)$ and the following hold.
\begin{itemize}
\item The restriction $A:=\AA\vert_{\ker(\BB)}$ with domain $\Dom(A)=\ker(\BB)$ generates a strongly continuous semigroup $\T$ 
 on $X$.
\item The operator $\BB\in \Lin(\Dom(\AA),U\times U_{d,3})$ has a bounded right inverse, i.e., there exists $\BB^r\in\Lin(U\times U_{d,3},\Dom(\AA))$ such that $\BB\BB^r = I$.
\end{itemize}
The \emph{transfer function} $P:\rho(A)\to \Lin(U\times U_{d,3},U)$ of $(\BB,\AA,\CC)$ is defined so that $P(\gl)(u,w)^\top = \CC x$, where $x\in \Dom(\AA)$ is the unique solution of the equations $(\gl-\AA)x=0$ and $\BB x = (u,w)^\top$.

\begin{definition}
\label{def:BCSStates}
Let $(\BB,\AA,\CC)$ be a boundary node on $(U\times U_{d,3},X,U)$, let $B_{d,1}\in \Lin(U_{d,1},X)$ and $\kappa\ge 0$, and consider~\eqref{eq:BCS} with control input of the form $u=u_1-\kappa y$ and $\wdist=(\wdistk{1},\wdistk{2},\wdistk{3})^\top$.

A tuple $(x,u_1,\wdist,y)$ 
is called a \emph{classical solution} of \eqref{eq:BCS} on $\zinf$ if
     $x\in C^1(\zinf;X)$,
$u_1,y\in C(\zinf;U)$, and
 $\wdist\in C(\zinf;U_d)$
 and if $x(t)\in D(\AA)$ 
 and
the identities in~\eqref{eq:BCS} hold for every $t\geq 0$.

A tuple $(x,u_1,\wdist,y)$
is called a \emph{generalised solution} of~\eqref{eq:BCS} on $\zinf$ if 
     $x\in C(\zinf;X)$, 
    $u_1,y\in\Lploc[2](0,\infty;U)$, and $\wdist\in\Lploc[2](0,\infty;U_d)$ and if
     there exists a sequence $(x^k,u_1^k,\wdist^k,y^k)_k$ of classical solutions of~\eqref{eq:BCS} on $\zinf$ such that
$
(\Pt[\tau] x^k,\Pt[\tau] u_1^k,$ $\Pt[\tau]\wdist^k,\Pt[\tau] y^k)^\top\to (\Pt[\tau] x,\Pt[\tau] u_1,\Pt[\tau]\wdist,\Pt[\tau] y)^\top
$
as $k\to \infty$ 
in $C([0,\tau];X)\times \L^2(0,\tau;U) \times \L^2(0,\tau;U_d)\times \L^2(0,\tau;U)$ for every $\tau>0$ .
\end{definition}

The boundary node $(\BB,\AA,\CC)$ is called \emph{well-posed} if there exist $\tau,M_\tau>0$ such that all classical solutions of~\eqref{eq:BCS} with $\phi = \mathrm{id}$ and with $\wdistk{1}(t)\equiv 0$ and $\wdistk{2}(t)\equiv 0$ satisfy~\citel{JacZwa12book}{Ch.~13}
\eq{
\norm{x(\tau)}_X^2 + \int_0^\tau\norm{y(t)}_U^2d t \leq M_\tau \left( \norm{x_0}_X^2 + \int_0^\tau \norm{u(t)}_U^2 + \norm{\wdistk{3}(t)}_{U_{d,3}}^2 d t \right).
}

We solve the \SORP\ under the following assumptions on~\eqref{eq:BCS}.

\begin{assumption}
\label{ass:ORPBCSass}
Let $X$, $U$, $U_{d,1}$, and $U_{d,3}$ be Hilbert spaces and
suppose that the following hold.
\begin{itemize}
\item[\textup{(a)}] 
$B_{d,1}\in \Lin(U_{d,1},X)$ and 
$(\BB,\AA,\CC)$ with $\BB=\pmatsmall{\BB_c\\\BB_d}$ is a well-posed boundary node on $(U\times U_{d,3},X,U)$.
Moreover, 
\eq{
\re \iprod{\AA x}{x}_X \leq \re \iprod{\BB_c x}{\CC x }_U, \qquad x\in \ker(\BB_d).
}
\item[\textup{(b)}] Either $\kappa=0$ and the semigroup generated by $A$ is strongly stable,
\item[\textup{(b')}] 
 or, alternatively, $\kappa>0$, 
$\norm{r_k}<\gd_k $ for all $k\in \List{p}$, and the semigroup generated by
 $A^\kappa = \AA\vert_{\ker(\BB_c + \kappa \CC)\cap \ker(\BB_d)}$ with domain $\Dom(A^\kappa) = \ker(\BB_c + \kappa \CC)\cap \ker(\BB_d)$
is strongly stable.
\end{itemize}
\end{assumption}

The control law we use is again of the form $u(t)=-\kappa (y(t)-\yref(t)) +\ureg(t)$, 
where $\ureg$ is defined in~\eqref{eq:ureg}  with coefficients of the form~\eqref{eq:uregcoeff}. Now the transfer function values $P^\kappa_c(\pm i\gw_k)$ and $P^\kappa_d(\pm i\gw_k)$ are determined by the parameters of the boundary control system~\eqref{eq:BCS} so that $y = P_c^\kappa (\gl )u+ P_d^\kappa (\gl)w$ with $u\in U$ and $w=(w_1,w_2,w_3)^\top\in U_d $ if and only if $y=\CC x$, where $x\in \Dom(\AA)$ is the unique solution of the abstract elliptic problem
\eqn{
\label{eq:BCSellipticprob}
\begin{cases}
(\gl-\AA)x = B_{d,1}w_1\\
(\BB_c+\kappa \CC)x = u+w_2\\
\BB_d  x=  w_3.
\end{cases}
}
The following theorem is the main result of this section.

\begin{theorem}
\label{thm:ORPforBCS}
Suppose that Assumption~\textup{\ref{ass:ORPBCSass}} holds with
 $\kappa\ge 0$ and that $\pm i\gw_k\in \rho(A^\kappa)$ and $P_c^\kappa(\pm i\gw_k)$ is boundedly invertible for every $k\in \List{q}$.  If $a_0\neq 0$ or $c_0\neq 0$, assume also that $0\in \rho(A^\kappa)$ and $P_c^\kappa(0)$ is boundedly invertible.
Let $\ureg$ be as in~\eqref{eq:ureg}--\eqref{eq:uregcoeff}.
For every $x_0\in X$ and for every $\wdist$ and $\yref$ of the form~\eqref{eq:yrefwdist} 
the
 system resulting from the control law
\eq{
u(t)= -\kappa e(t) + \ureg(t), \qquad t\geq 0
}
with $e=y-\yref$
has a unique generalised solution $(x,\ureg + \kappa \yref,\wdist,y)$ satisfying $x(0)=x_0$.
If the coefficients $(a_k,b_k,c_k,d_k)_k$ of $\yref$ and $\wdist$ in~\eqref{eq:yrefwdist} are such that
 $\ureg(t)\in \Ulim^\gd$ for some $\gd>0$ and for all $t\geq 0$, 
then for every $x_0\in X$ the state trajectory $x$ is uniformly bounded and
the following hold.
\begin{itemize}
\item If $\kappa>0$, then 
there exists a set $\Omega\subset \zinf$ of finite measure such that 
 $e\in \Lp[2](\zinf\setminus\Omega;U)$ and $ e\in \Lp[1](\Omega;U)$.
\item If $\kappa = 0$, then $\norm{e}_{\Lp[2](t,t+1)}\to0 $ as $t\to\infty$. If, in addition, $\T$ is exponentially stable, then $t\mapsto e^{\ga t}e(t)\in \Lp[2](0,\infty;U)$ for some $\ga>0$.
\end{itemize}
If $x_0\in \Dom(\AA)$, $\wdist = (\wdistk{1},\wdistk{2}, \wdistk{3})^\top$, and $\yref$ are such that
 $\BB_c  x_0=\phi(\ureg(0)  + \kappa \yref(0)-\kappa \CC x_0 )+\wdistk{2}(0)$ and 
$\BB_d x_0 = \wdistk{3}(0)$, then 
$(x,\ureg+\kappa\yref,\wdist,y)$ is a classical solution of~\eqref{eq:BCS} and $e\in \Hloc{1}(0,\infty;U)$.
\end{theorem}

\begin{proof}
We will rewrite~\eqref{eq:BCS} in the form~\eqref{eq:SysMain} with $\wdist = (\wdistk{1},\wdistk{2},\wdistk{3})^\top$ and apply \cref{thm:ORPmain}.
This is possible due to the results in~\cite{MalSta06} which 
 show that the boundary node $(\BB,\AA,\CC)$ gives rise to a system node.
 More precisely,
\citel{MalSta06}{Thm.~2.3} shows that
the operator 
\eq{
S_0= \pmat{\AA\\ \CC}\pmat{I\\ \BB}\inv , \qquad \Dom(S_0)= \ran \left( \pmat{I\\ \BB} \right)
}
is a system node on $(U\times U_{d,3},X,U)$.
If we define
$U_d = U_{d,1}\times U\times U_{d,3}$ and 
 $S: \Dom(S)\subset X\times U\times U_d\to X\times U$ so that $\Dom(S)=\setm{(x,u,w)^\top\in X\times U\times U_d}{w=(w_1,w_2,w_3)^\top , \ (x,u+w_2,w_3)^\top \in \Dom(S_0)}$ and
\eq{
S \pmat{x\\u\\w} = S_0 \pmat{x\\ u+w_2\\ w_3}+\pmat{B_{d,1}w_1\\0}, 
}
for all $(x,u,w)^\top \in \Dom(S)$ with $w=(w_1,w_2,w_3)^\top $, then $S$ is a system node on $(U\times U_d,X,U)$.
The definition of $S$ implies that $(x,u,w)^\top\in \Dom(S)$ with $w=(w_1,w_2,w_3)^\top$ if and only if $x\in \Dom(\AA)$, $\BB_c x=u+w_2$ and $\BB_d x=w_3$.
In this case we have
 $S(x,u,w)^\top = (\AA x + B_{d,1}w_1,\CC x)^\top $.
Definitions~\ref{def:SysNodeStates} and~\ref{def:BCSStates} and the relationship between $(\BB,\AA,\CC)$ and $S$ imply that the classical and generalised solutions of~\eqref{eq:SysGenInputs} coincide with the classical and generalised solutions, respectively, of~\eqref{eq:BCS} with $u=u_1-\kappa y$.
Our assumption that $(\BB,\AA,\CC)$ is well-posed and~\citel{Sta05book}{Thm.~4.7.13} imply that $S$ is a well-posed system node.

We will now show that $S=\pmatsmall{\AB\\\CD}$ satisfies \cref{ass:ORPass}.
 \cref{ass:ORPBCSass}(a)
implies that if $(x,u,0)^\top \in \Dom(S)$, then 
$\BB_c x=u$,
$\BB_d x=0$, and
\eq{
\re \Iprod{\AB \pmatsmall{x\\u\\0}}{x}
= 
\re \iprod{\AA x}{x} 
\leq \re \iprod{\CC x}{ \BB_c x}
=\re \Iprod{\CD \pmatsmall{x\\u\\0}}{u} .
}
In addition,
if we denote the transfer function of $S_0$ by $[P_c,P_{d,0}]$, then
the transfer function $P$ of $S$ has the form $P=[P_c,P_d]$ where $P_d(\gl)=[C(\gl-A)\inv B_{d,1},P_c(\gl),P_{d,0}(\gl)]$ for all $\gl\in\rho(A)$. The definition of $S_0$ also implies that $P_c$ coincides with the transfer function of the boundary node $(\BB_c,\AA\vert_{\ker(\BB_d)},\CC)$. Our assumption that 
$\re \iprod{\AA x}{x}_X \leq \re \iprod{\BB_c x}{\CC x }_U$ for all $ x\in \ker(\BB_d)$ and~\citel{HasPau25arxiv}{Lem.~5.1} imply that $P_c(\gl)\ge c_\gl I$ for some $\gl,c_\gl>0$. Thus $S$ satisfies \cref{ass:ORPBCSass}(a).

We note that 
 $\CD = [\CC,0]\vert_{\Dom(S)}$  by~\citel{MalSta06}{Thm.~2.3(iv)}.
Using this it is straightforward to show that the operator $A^\kappa$ in \cref{ass:ORPass} 
 is exactly the operator $A^\kappa= \AA\vert_{\ker(\BB_c+\kappa \CC)\cap \ker(\BB_d)}$ in \cref{ass:ORPBCSass}(b'). Therefore either (b) or (b') of \cref{ass:ORPass} holds.
If $x_0\in \Dom(\AA)$, $\wdist = (\wdistk{1},\wdistk{2}, \wdistk{3})^\top$, and $\yref$ are such that
 $\BB_c  x_0=\phi(\ureg(0)  + \kappa \yref(0)-\kappa \CC x_0 )+\wdistk{2}(0)$ and 
$\BB_d x_0 = \wdistk{3}(0)$, then 
the definition of $S$ implies that $v=\CC x_0\in U$ satisfies
 $(x_0,\phi(\ureg(0)-\kappa (v- \yref(0))),\wdist(0))^\top \in \Dom(\Stot)$ and $v=\CD(x_0,\phi(\ureg(0)-\kappa (v- \yref(0))),\wdist(0))^\top$.
Therefore the claims follow directly from \cref{thm:ORPmain}.
\end{proof}

\section{Saturated Output Regulation for PDE Systems}
\label{sec:PDEcases}

\subsection{A Two-Dimensional Heat Equation}
\label{sec:Heat}

In this section we consider output tracking for a two-dimensional heat equation on the square $Q = [0,1]\times [0,1]$. 
The system has
two boundary inputs $u(t)=(u_1(t),u_2(t))^\top$, two boundary outputs $y(t)=(y_1(t),y_2(t))^\top$, one boundary disturbance $\wdist(t)$ and is determined by
    \eq{
      \pd{x}{t}(\xi,t) &= \Delta x(\xi,t), \qquad x(\xi,0)=x_0(\xi) \\
      \pd{x}{n}(\xi,t)\vert_{\Gamma_1} &= \phi_1(u_1(t))+\wdist(t), \ 
      \pd{x}{n}(\xi,t)\vert_{\Gamma_2} = \phi_2(u_2(t)), \
      \pd{x}{n}(\xi,t)\vert_{\Gamma_0} = 0\\
      y_1(t) &= \int_{\Gamma_1}x(\xi,t)d\xi, \qquad
      y_2(t) = \int_{\Gamma_2}x(\xi,t)d\xi.
    }
Here $\phi_1:\C\to \C$ and $\phi_2:\C\to \C$  are two saturation functions such that $\phi_k(u)=u$ when $\abs{u-\rk}\le \gd_k$ and $\phi_k(u)=\rk+\gd_k \abs{u-\rk}\inv (u-\rk)$ when $\abs{u-\rk}>\gd_k$ for $k=1,2$.
We assume that $\abs{r_k}<\gd_k$, $k=1,2$.
The parts $\Gamma_0$, $\Gamma_1$, and $\Gamma_2$ of the boundary $\partial Q$ are defined so that
  $\Gamma_1 = \setm{\xi=(\xi_1,0)^\top}{0\leq \xi_1\leq 1/2}$,
  $\Gamma_2 = \setm{\xi=(\xi_1,1)^\top}{1/2\leq \xi_1\leq 1}$, 
  $\Gamma_0 = \partial Q \setminus (\Gamma_1 \cup \Gamma_2)$. 
By~\cite[Cor. 1]{ByrGil02} the heat equation without the saturation functions defines a well-posed linear system (which is also \emph{regular} by~\citel{ByrGil02}{Cor.~02}).
This allows us to reformulate the equation as an abstract system of the form~\eqref{eq:SysMain}. 
By~\citel{ByrGil02}{Sec.~7} the associated system node
$S$ on $(U\times U_d,X,U)$ with $X=\Lp[2](Q)$, $U=\C^2$, $U_d=\C$
 is defined by
\eq{
\Dom(S) &= \Bigl\{(x,u,w)^\top\in H^1(Q)\times U\times U_d\,\Bigm|\, \Delta x\in \Lp[2](Q), \ u=(u_1,u_2)^\top,\\
&\hspace{2cm}
\frac{\partial x}{\partial n}=u_1+w \mbox{ on} \ \Gamma_1,\ 
\frac{\partial x}{\partial n}=u_2 \mbox{ on} \ \Gamma_2,\ 
\frac{\partial x}{\partial n}=0 \mbox{ on} \ \Gamma_0
  \Bigr\}\\
S\pmat{x\\ u\\w}
&=
\pmat{
A\& B\\
C\& D 
}
\pmat{x\\ u\\w} = 
\pmat{
\Delta x\\
 \int_{\Gamma_1} x(\xi)d \xi\\
 \int_{\Gamma_2} x(\xi)d \xi
}
, \qquad \pmat{x\\u\\w}\in \Dom(S).
}
In addition, we define the function $\phi:U\to U$ by $\phi(u)=(\phi_1(u_1),\phi_2(u_2))^\top$ for $u=(u_1,u_2)^\top\in U$.
The system node is well-posed by~\citel{ByrGil02}{Cor.~1}.

\hspace{-3ex}
We consider output tracking and disturbance rejection for signals $\yref=(\yref^1,\yref^2)^\top$ and $\wdist$ of the form~\eqref{eq:yrefwdist}.
The control law again consists of the negative error feedback term and the part $\ureg = (\ureg^1,\ureg^2)^\top$ computed based on the transfer function values of the system and the parameters of the signals in~\eqref{eq:yrefwdist} using~\eqref{eq:ureg} and~\eqref{eq:uregcoeff}.
The values $P^\kappa_c(\pm i\gw_k)$ and $P^\kappa_d(\pm i\gw_k)$
can be computed by taking a formal Laplace transform of the heat system with $\phi = \id$ and under negative output feedback $u=-\kappa y+\tilde u$~\citel{JacZwa12book}{Sec.~12.1}.
Because of this,
 $y = P_c^\kappa (\pm i\gw_k)u+ P_d^\kappa (\pm i\gw_k)w$ with $u=(u_1,u_2)^\top\in U$ and $w\in U_d $ if and only if $y=(\int_{\Gamma_1}x(\xi)d\xi,\int_{\Gamma_2}x(\xi)d\xi )^\top$, where $x\in \Lp[2](Q)$ with $\Delta x\in \Lp[2](Q)$ is the unique solution of the boundary value problem
\eq{
\pm i\gw_k x(\xi) &= \Delta x(\xi), \quad\ \xi\in Q, \qquad 
&&\pd{x}{n}(\xi)\vert_{\Gamma_0} = 0,
\\
\pd{x}{n}(\xi)\vert_{\Gamma_1} &= -\kappa\int_{\Gamma_1}x(\xi)d\xi + u_1 + 
w, \quad
&&\pd{x}{n}(\xi)\vert_{\Gamma_2} = \-\kappa\int_{\Gamma_2}x(\xi)d\xi +u_2.
}
We have $i\gw\in\rho(A)$ for $\gw\in \R\setminus \set{0}$. Thus it follows from Section~\ref{sec:ComputingTheControl} that to compute $f_k$ and $g_k$ in~\eqref{eq:uregcoeff} for $k\in \List{q}$ it suffices to solve the above system (numerically) for $\kappa=0$.  For $\gw_0=0$ the system is solved with $\kappa>0$.

\begin{proposition}
\label{prp:ORPheat}
Suppose that
 $\kappa>0$.
The system
 with the control law
\eq{
u(t)= -\kappa e(t) + \ureg(t), \qquad t\geq 0
}
solves the \SORP\ for those reference and disturbance signals of the form~\eqref{eq:yrefwdist}
for which $\ureg(t)\in \Ulim^\gd$ for some $\gd>0$ and for all $t\geq 0$.
For every $x_0\in X$
 the system~\eqref{eq:SysGenInputs} has a well-defined generalised solution
$(x,\ureg+\kappa \yref,\wdist,y)$ 
satisfying $x(0)=x_0$.
Moreover, for every $x_0\in X$ the state trajectory $x$ is uniformly bounded and there exists a set $\Omega\subset \zinf$ of finite measure such that 
\eq{
 e\in \Lp[2](\zinf\setminus\Omega;U)
\qquad \mbox{and} \qquad
e\in \Lp[1](\Omega;U) .
}
If $x_0\in H^1(Q)$, $\wdist$, and $\yref$ are such that $\Delta x_0\in \Lp[2](Q)$ and if
\eq{
\pd{x_0}{n} &=  \phi_1(\ureg^1(0)-\kappa (y_1-\yref^1(0)))+\wdist(0) \quad \ \mbox{on} \quad \Gamma_1\\
\pd{x_0}{n} &=  \phi_2(\ureg^2(0)-\kappa (y_2-\yref^2(0))) \quad \ \mbox{on} \quad \Gamma_2,
}
where
$y_1 = \int_{\Gamma_1}x_0(\xi)d\xi$ and $y_2 = \int_{\Gamma_2}x_0(\xi)d\xi$, 
then
$(x,\ureg+\kappa \yref,\wdist,y)$ is a classical solution of~\eqref{eq:SysGenInputs} and $e\in \Hloc{1}(0,\infty;U)$.
\end{proposition}

\begin{proof}
We begin by showing that $S$ and $\phi$ satisfy \cref{ass:ORPass}.
The well-posedness of $S$ follows from~\citel{ByrGil02}{Cor.~1}.
If $x\in X$ and $u\in U$ are such that $(x,u,0)\in \Dom(S)$, then integration by parts shows that
\eq{
\re \Iprod{A\& B \pmatsmall{x\\ u\\0} }{x}_X
&= \re \iprod{\Delta x}{x}_{\Lp[2](Q)}\\
&= \re\left(\int_{\Gamma_1}u_1 \overline{x(\xi)} d \xi
+ \int_{\Gamma_2}u_2 \overline{x(\xi)} d \xi\right)
 - \norm{\nabla x}_{\Lp[2](Q)}^2\\
&\le \re \Iprod{ C\& D \pmatsmall{x\\u\\0}}{u} .
}
The last property of \cref{ass:ORPass}(a), requiring that $\re P_c(i\gw_k)\ge c_\gl I$ for some $\gl,c_\gl>0$, is shown below.
The exponential stability of the semigroup generated by $A^\kappa$ in part (b') of \cref{ass:ORPass} can be shown exactly as in the proof of~\citel{HasPau25arxiv}{Prop.~6.6}.

We will now show that $P_c^\kappa(i\gw)$ is nonsingular for every $\gw\in\R$. 
To this end, let $S^\kappa$
 be the system node associated with the well-posed linear system
$(\T^\kappa,\Phi^\kappa,\Psi^\kappa,\F^\kappa)$.
Then the transfer function $P^\kappa = [P_c^\kappa,P_d^\kappa]$ coincides with the transfer function of $S^\kappa$ on $\overline{\C_+}$.
Fix $\gw\in\R$ and let $u\in U$ be such that $P_c^\kappa(i\gw)u=0$. 
 If we define $x=(i\gw-A^\kappa)\inv B^\kappa \pmatsmall{u\\0}$, then $(x,u,0)^\top\in \Dom(S^\kappa)$ and $P_c^\kappa(i\gw)u= (\CD)_\kappa \pmatsmall{u\\0}$.
Thus 
$(i\gw x,0)^\top = S^\kappa (x,u,0)^\top$,
and~\citel{Sta05book}{Def.~7.4.2} further implies that
$(x,u,0)^\top\in \Dom(S)$ and
$(i\gw x,0)^\top = S (x,u,0)^\top$.
Similarly as above, integration by parts implies that 
\eq{
0=\re \iprod{i\gw x}{x}
= \re \iprod{\Delta x}{x}_{\Lp[2]}
= \re \Iprod{\CD \pmatsmall{x\\u\\0}}{u} - \norm{\nabla x}_{\Lp[2](Q;\C^2)}^2,
}
which implies that $\nabla x=0$. Thus $x$ is a constant function, and since $\int_{\Gamma_1}x(\xi)d\xi=\int_{\Gamma_2}x(\xi)d\xi=0$ by assumption, we must have $x=0$, and consequently also $u=0$. Since $\gw\in\R$ and $u\in U$ were arbitrary so that $P_c^\kappa(i\gw)u=0$, we conclude that    $P_c^\kappa(i\gw)$ is nonsingular for every $\gw\in\R$.

The analysis above in particular shows that $P_c^\kappa(0)$ is nonsingular. Since $0\in\rho(A^\kappa)$, by continuity there exists a small $\gl>0$ such that
 $P_c^\kappa(\gl)$ is nonsingular.
We have $\gl\in \rho(A)\cap \rho(A^\kappa)$, and thus $I-\kappa P_c^\kappa(\gl)$ is nonsingular and $P_c(\gl)=P_c^\kappa(\gl)(I-\kappa P_c^\kappa(\gl))\inv$ by~\citel{Sta05book}{Lem.~7.1.7 \& 7.4.4}. Since $P_c(\gl)^\ast = P_c(\gl)$ and $P_c(\gl)\ge 0$, 
the previous formula also implies that $\re P_c(\gl)=P_c(\gl)$ is nonsingular, and therefore $\re P_c(\gl)\ge c_\gl I$ for some $c_\gl>0$.

Finally, suppose that $x_0\in H^1(Q)$, $\yref$ and $\wdist$ satisfy the additional conditions in the statement of the theorem. 
If we define $v=(y_1,y_2)^\top$, then the definition of $S$ and the assumptions on $x_0$, $\yref$, and $\wdist$ imply that 
 $(x_0,\phi(\ureg(0)-\kappa (v- \yref(0))),\wdist(0))^\top \in \Dom(\Stot)$ and $v=\CD(x_0,\phi(\ureg(0)-\kappa (v- \yref(0))),\wdist(0))^\top$.
Thus the claims follow from \cref{thm:ORPmain}.
\end{proof}

We illustrate the controller design by considering the tracking and rejection of the signals $\yref = (\yref^1,\yref^2)^\top$ with
$\yref^1(t) = 1+\sin(\pi t)$, $\yref^2(t)=2+0.5\cos(\pi t)+\cos(3\pi t)$, and $\wdist(t)=2+3\cos(5\pi t)$.
These signals are of the form~\eqref{eq:yrefwdist} with the choices $(\gw_1,\gw_2,\gw_3)^\top=(\pi,3\pi,5\pi)^\top$, $a_0=(1,2)^\top$, $a_1=(0,1/2)^\top$, $a_2=(0,1)^\top$, $a_3=(0,0)^\top$, $b_1=(1,0)^\top$, $b_2=b_3=(0,0)^\top$, 
$c_0=2$, $c_1=c_2=c_3=0$, $d_1=d_2=0$ and $d_3=3$.
We set $r_1=-4.5$, $r_2=1.75$, $\delta_1=10.5$, and $\gd_2=15.75$.
In the simulation, the heat system is approximated using a truncated basis of eigenfunctions of the two-dimensional Laplacian with $31\times 31$ eigenmodes.
The simulation results for $\kappa=3$ and the intial state $x_0(\xi)=-10(1+\cos(\pi(1-\xi_1)))(1-\cos(2\pi\xi_2)/4)$
are depicted in \cref{fig:Heat2D}.  In computing the transfer function values in the parameters~\eqref{eq:uregcoeff} of $\ureg$
we use a higher-dimensional approximation of the heat system with $41\times 41$ eigenmodes.

\begin{figure}[h!]
\begin{minipage}[c]{0.4\linewidth}
\includegraphics[width=\linewidth]{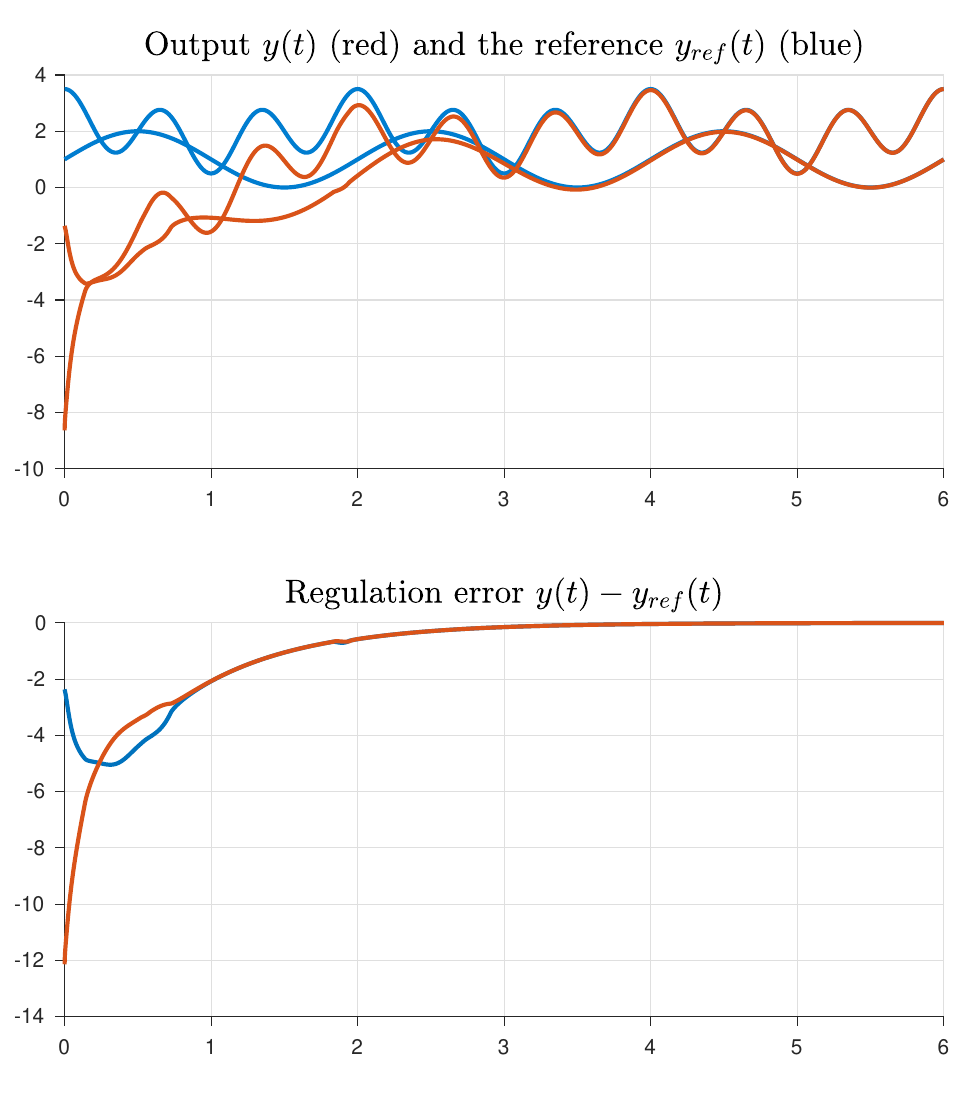}
\end{minipage}
\hfill
\begin{minipage}[c]{0.59\linewidth}
\begin{center}
\includegraphics[width=0.9\linewidth]{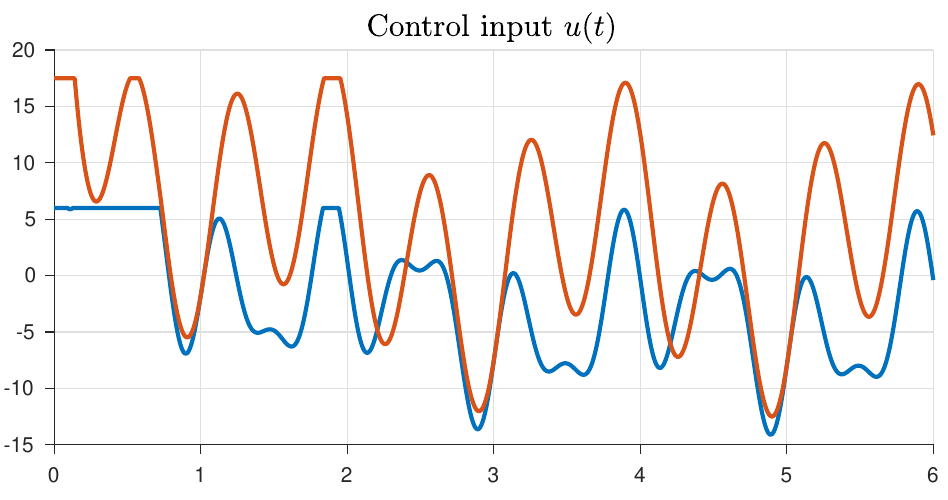}
\end{center}
\end{minipage}
\caption{Simulation of the controlled two-dimensional heat equation.}
\label{fig:Heat2D}
\end{figure}

\subsection{A One-Dimensional Wave Equation}
\label{sec:Wave}

We consider a one-dimensional wave equation on $[0,1]$ defined by
\eq{
  \rho(\xi) \pd[2]{v}{t}(\xi,t) &=  \pdb{\xi}\left( T(\xi)\pd{v}{\xi}(\xi,t) \right), \\
      -T(0)\pd{v}{\xi}(0,t) &= \phi(u(t)), \qquad 
      T(1)\pd{v}{\xi}(1,t) = \wdistk{3}(t), \\
      y(t) &= \pd{v}{t}(0,t) 
    }
with initial states $ v(\xi,0)=v_0(\xi)$  and $ \pd{v}{t}(\xi,0)=v_1(\xi)$.
Here $v(\xi,t)$ is the deflection of the string at position $\xi\in [0,1]$ and at time $t\ge 0$, $\rho$ is the mass density of the string, and $T$ is the Young's modulus.
Moreover, $\phi$ is the scalar saturation function such that $\phi(u)=\rk[1]+\gd_1(u-\rk[1])/\max \set{\gd_k,\norm{u-\rk[1]}}$ for some $\gd_1>0$ and $\rk[1]\in \C$ with $\abs{\rk[1]}<\gd_1$.

We assume that $\rho,T\in C^2([0,1])$, and that $\rho(\xi)>0 $ and $T(\xi)>0$ for $\xi\in[0,1]$.
Similarly as in~\citel{JacZwa12book}{Ex.~13.2.4},
the wave equation can be reformulated as an abstract boundary control system~\eqref{eq:BCS} with state $x(t)=(\rho(\cdot)\pd{v}{t}(\cdot,t),\pd{v}{\xi}(\cdot,t))^\top$ on the spaces $(U\times U_{d,3},X,U)$ where $U=\C$, $U_{d,3}=\C$, and
$X=\Lp[2](0,1)\times \Lp[2](0,1)$ with inner product $\iprod{(f_1,g_1)^\top}{(f_2,g_2)^\top}_X=\iprod{\rho\inv f_1}{f_2}_{\Lp[2]}+ \iprod{T g_1}{g_2}_{\Lp[2]}$.
 To achieve this, we define  $\AA: \Dom(\AA)\subset X\to X$, and $\BB_c,\BB_d,\CC\in \Lin(\AA,\C)$ so that 
$\Dom(\AA) =  H^1(0,1)\times H^1(0,1)$, and
\eq{
\AA \pmat{x_1\\ x_2} &= \pmat{ (Tx_2)'\\ (\frac{1}{\rho} x_1)'}, \qquad
\BB_c \pmat{x_1\\x_2} = -T(0)x_2(0), \\
\BB_d \pmat{x_1\\x_2} &= T(1)x_2(1), \qquad
\CC \pmat{x_1\\x_2} = \frac{1}{\rho(0)} x_1(0).
}
We have $U_{d,1}=\set{0}$, $\wdistk{1}=0$, $\wdistk{2}=0$, and $\wdist = (0,0,\wdistk{3})^\top$ in~\eqref{eq:BCS}. The initial state is given by $x_0=(\rho v_1,v_0')^\top$.
We have from~\citel{JacZwa12book}{Sec.~13.5} that $(\BB,\AA,\CC)$ with $\BB=\pmatsmall{\BB_c\\\BB_d}$ is a well-posed boundary node.

We consider 
 $\yref$ and $\wdist$ of the form~\eqref{eq:yrefwdist} with $a_0=0$ and $c_0=0$.
The control law consists of the negative error feedback term and  $\ureg $ in~\eqref{eq:ureg}.
In computing the parameters $f_k$ and $g_k$ of $\ureg$ in~\eqref{eq:uregcoeff}, the transfer function values can be computed by solving a static boundary value problem~\eqref{eq:BCSellipticprob} for each frequency $\gl = \pm i\gw_k$. 
More precisely, for $\gw\in\R$ and $u,w\in\C$  we have $y
= P_c^\kappa (i\gw)u+P_d^\kappa(i\gw)w$ 
where $y\in \C$ and $v\in H^2(0,1)$ satisfy
\eq{
  -\gw^2\rho(\xi) v(\xi) &=  \left( T(\xi)v'(\xi) \right)', \qquad \xi \in(0,1)\\
      -T(0)v'(0) +i\kappa\gw v(0) &=  u, \qquad 
      T(1)v'(1) = w, \\
      y &= i\gw v(0). 
    }
Since $\rho$ and $T$ are real-valued functions, we have $P(-i\gw_k)=\conj{P(i\gw_k)}$, it sufficies to solve the boundary value problem for $\gl = i\gw_k$ with $k\in \List{q}$.

\begin{proposition}
\label{prp:ORPwave}
Suppose that 
 $\kappa>0$ and
 $P_c^\kappa(i\gw_k)\neq 0$  for all $k\in \List{q}$.
The control law
\eq{
u(t)= -\kappa e(t) + \ureg(t), \qquad t\geq 0
}
solves the \SORP\ for those reference and disturbance signals of the form~\eqref{eq:yrefwdist} with $a_0=0$ and $c_0=0$
for which $\ureg(t)\in \Ulim^\gd$, $t\ge 0$, for some $\gd>0$.
For every $x_0\in X$ the system~\eqref{eq:SysGenInputs} has a well-defined generalised solution $(x,\ureg+\kappa \yref,\wdist,y)$ 
satisfying $x(0)=x_0$.
Moreover, for every $x_0\in X$ the state trajectory $x$ is uniformly bounded and there exists a set $\Omega\subset \zinf$ of finite measure such that 
\eq{
 e\in \Lp[2](\zinf\setminus\Omega)
\qquad \mbox{and} \qquad
e\in \Lp[1](\Omega) .
}
If $v_0\in H^2(0,1)$, $v_1\in H^1(0,1)$, $\wdist$, and $\yref$ are such that 
\eq{
      -T(0)v_0'(0) &= \phi(\ureg(0)-\kappa (v_1(0)-\yref(0))), \qquad 
      T(1)v_0'(1) = \wdistk{3}(0), 
}
then
$(x,\ureg+\kappa \yref,\wdist,y)$ is a classical solution of~\eqref{eq:SysGenInputs} and $e\in \Hloc{1}(0,\infty)$.
\end{proposition}

\begin{proof}
We begin by showing that $(\BB,\AA,\CC)$ satisfies \cref{ass:ORPBCSass}.
The well-posedness of $(\BB,\AA,\CC)$ follows from~\citel{JacZwa12book}{Sec.~13.5}.
If $(x_1,x_2)\in \ker(\BB_d)$, then $T(1)x_2(1)=0$ and integration by parts shows that
\eq{
\re \iprod{\AA x}{x}_X 
&= \re \left(\iprod{\rho\inv (Tx_2)'}{x_1}_{\Lp[2]}+ \iprod{T(\rho\inv x_1)'}{x_2}_{\Lp[2]}\right)\\
&=\re \left( \rho(1)\inv x_1(1) \conj{T(1) x_2(1)}- 
\rho(0)\inv x_1(0) \conj{T(0) x_2(0)} \right)
\\
&
= \re \iprod{\BB_c x}{\CC x }_U.
}
The operator $A^\kappa = \AA\vert_{\ker(\BB_c+\kappa \CC)\cap \ker(\BB_d)}$ is associated with a wave equation with damping at $\xi =0$. We have from~\cite{CoxZua95} that the semigroup generated by $A^\kappa$ is exponentially stable. Thus \cref{ass:ORPBCSass} holds. 
 Since $P_c^\kappa(-i\gw)=\conj{P_c^\kappa(i\gw)}$, $\gw\in\R$, we have
 $P_c^\kappa(\pm i\gw_k)\neq 0$ for all $k\in \List{q}$ by assumption.
Finally, we note that the additional assumptions for $v_0$ and $v_1$ coincide with the 
requirement that $x_0=(\rho v_1,v_0')$ and $\wdist=(0,0,\wdistk{3})^\top$  satisfy $x_0\in \Dom(\AA)$,
$\BB_c  x_0=\phi(\ureg(0)  + \kappa \yref(0)-\kappa \CC x_0 ) + \wdistk{2}(0)$, and 
$\BB_d x_0 = \wdistk{3}(0)$. 
Because of this, the claims follow from \cref{thm:ORPforBCS} and the proof is complete.
\end{proof}

We illustrate the controller design by studying
the wave system with $\rho(\xi)\equiv 1$ and $T(\xi)\equiv 1$ and by considering
$\yref(t) = \sin(\pi t)+\cos(3\pi t)$ and $\wdistk{3}(t)=0.5\cos(5\pi t)$.
These signals are of the form~\eqref{eq:yrefwdist} with the choices $(\gw_1,\gw_2,\gw_3)^\top=(\pi,3\pi,5\pi)^\top$, $a_0=a_1=0$, $a_2=1$, $a_3=0$, $b_1=1$, $b_2=b_3=0$, 
$c_0=0$, $c_1=c_2=c_3=0$, $d_1=d_2=0$ and $d_3=1/2$.
We set $r_1=0$ and $\gd_1=2.5$.
The wave system is approximated using a truncated basis of eigenfunctions of the undamped wave operator with $30$ eigenmodes.
The simulation results for $\kappa=3/4$ and the intial state $v_0(\xi)=1/2(1+\cos(3\pi \xi)+\cos(6\xi))$, $v_1(\xi)\equiv 0$
are depicted in \cref{fig:Wave1D}. 
In computing the transfer function values in the parameters~\eqref{eq:uregcoeff} of $\ureg$
we use a higher-dimensional approximation of the wave system with $40$ eigenmodes.

\begin{figure}[h!]
\begin{minipage}[c]{0.49\linewidth}
\includegraphics[width=\linewidth]{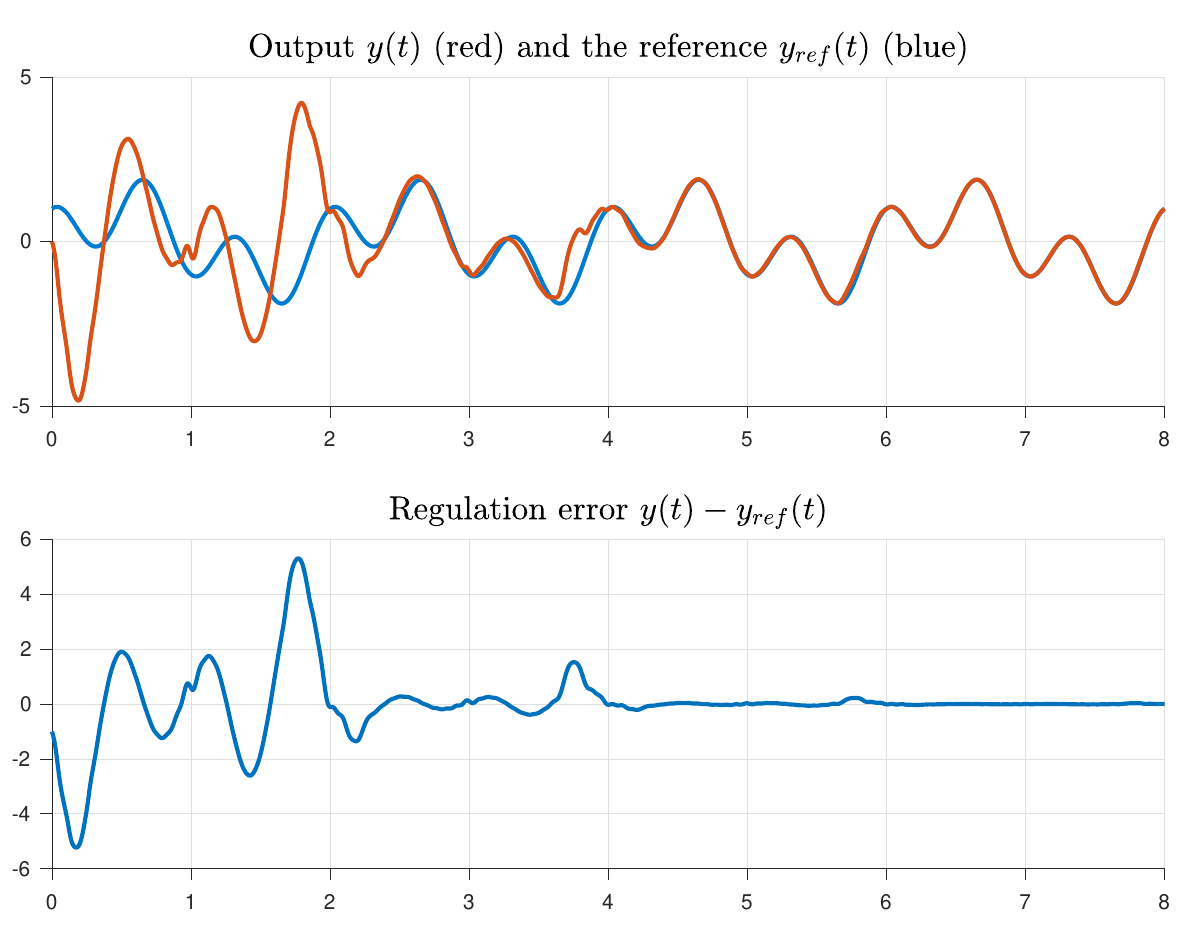}
\end{minipage}
\hfill
\begin{minipage}[c]{0.5\linewidth}
\includegraphics[width=\linewidth]{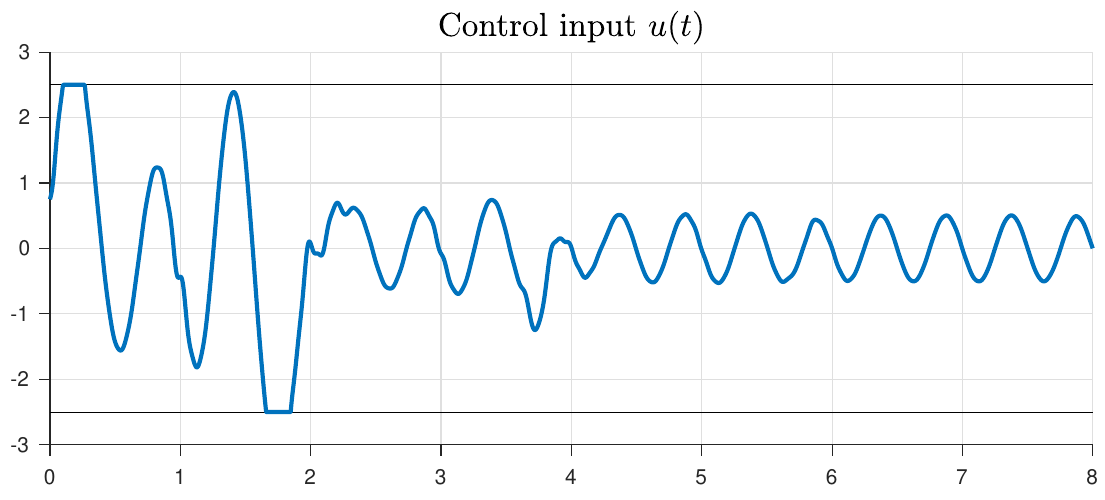}
\end{minipage}
\caption{Simulation of the one-dimensional wave equation.}
\label{fig:Wave1D}
\end{figure}

\end{document}